\documentclass[11pt,leqno,a4paper]{article}

\usepackage[T1]{fontenc}

\usepackage[english,italian]{babel}

\usepackage{exscale,ifthen,amsfonts,amsmath,amscd,amssymb,amsthm}

\usepackage{natbib}

\usepackage{geometry, tikz}

\usepackage{enumerate,color}

\allowdisplaybreaks[2]

\DeclareMathOperator{\Id}{Id}

\DeclareMathOperator{\Prb}{\mathbf{P}}
\DeclareMathOperator{\Mean}{\mathbf{E}}

\DeclareMathOperator{\argmin}{argmin}

\numberwithin{equation}{section}

\begin{document}

\newcommand{\trans}[1]{{#1}^\mathsf{T}}

\newcommand{\BB}[1]{\boldsymbol{#1}}
\newcommand{\prbms}[2][]{\mathcal{P}_{#1}(#2)}
\newcommand{\Borel}[1]{\mathcal{B}(#1)}


\newcounter{hypcount}
\newenvironment{hypenv}{\renewcommand{\labelenumi}{(A\arabic{enumi})}\begin{enumerate}\setcounter{enumi}{\value{hypcount}}}{\setcounter{hypcount}{\value{enumi}}\end{enumerate}}

\newcounter{hypcount2}
\newenvironment{Hypenv}{\renewcommand{\labelenumi}{\textsc{(H\arabic{enumi})}}\begin{enumerate}\setcounter{enumi}{\value{hypcount2}}}{\setcounter{hypcount2}{\value{enumi}}\end{enumerate}}

\newenvironment{enumrm}{\begin{enumerate}\renewcommand{\labelenumi}{\textup{(\roman{enumi})}}}{\end{enumerate}}


\newcommand{\hypref}[1]{(A\ref{#1})}
\newcommand{\hyprefall}{(A1)\,--\,(A\arabic{hypcount})}
\newcommand{\Hypref}[1]{(H\ref{#1})}
\newcommand{\Hyprefall}{(H1)\,--\,(H\arabic{hypcount2})}

\newcounter{dummy}
\newcommand{\hyprefallbutlast}{\setcounter{dummy}{\value{hypcount}}\addtocounter{dummy}{-1}(A1)\,--\,(A\arabic{dummy})}

\newcommand{\rmref}[1]{\setcounter{dummy}{\ref{#1}}(\roman{dummy})}


\theoremstyle{plain}
\newtheorem{thrm}{Theorem}[section]
\newtheorem{prop}{Proposition}[section]
\newtheorem{lemma}{Lemma}[section]
\newtheorem{crll}[thrm]{Corollary}

\theoremstyle{definition}
\newtheorem{defn}{Definition}[section]

\theoremstyle{remark}
\newtheorem{case}{Case}[section]
\newtheorem{rem}{Remark}[section]
\newtheorem{exmpl}{Example}[section]

\renewcommand{\phi}{\varphi}
\renewcommand{\epsilon}{\varepsilon}


\selectlanguage{english}

\title{Correlated equilibria and mean field games: \\ a simple model}

\author{Luciano Campi\footnote{Department of Mathematics ``Federigo Enriques'', University of Milan, Via Saldini 50, 20133 Milan, Italy.} \and Markus Fischer\footnote{Department of Mathematics ``Tullio Levi-Civita'', University of Padua, via Trieste 63, 35121 Padova, Italy.}}

\date{April 13, 2020; last revision June 3, 2021}

\maketitle

\begin{abstract}
In the context of simple finite-state discrete time systems, we introduce a generalization of mean field game solution, called correlated solution, which can be seen as the mean field game analogue of a correlated equilibrium. Our notion of solution is justified in two ways: We prove that correlated solutions arise as limits of exchangeable correlated equilibria in restricted (Markov open-loop) strategies for the underlying $N$-player games, and we show how to construct approximate $N$-player correlated equilibria starting from a correlated solution to the mean field game.
\end{abstract}

\medskip
 

\section{Introduction}\label{SectIntro}

Correlated equilibria are generalizations of Nash equilibria that allow for correlation between players' strategies. In this paper, we consider correlated equilibria for a simple class of symmetric finite horizon $N$-player games and their natural mean field game counterpart as the number of players $N$ goes to infinity.

\emph{Mean field games} (\mbox{MFGs}, for short), independently introduced by \citet{huangetal06} and \citet{lasrylions07}, arise as limit systems for certain symmetric stochastic $N$-player games with mean field interaction as the number of players $N$ tends to infinity. Each player interacts with her competitors only via the empirical distribution of their positions so that, when $N \to \infty$, one expects the empirical distribution to converge to the law of the ``representative player'' (Law of Large Numbers or Propagation of Chaos). In the limiting \mbox{MFG}, the ``representative player'' reacts optimally to the behavior of the population, which in turn should arise at equilibrium by aggregation of all identical players' best responses. For a thorough treatment of \mbox{MFG} theory from a probabilistic perspective we refer to the two-volume book by \citet{carmonadelarue18}. 

A rigorous connection between \mbox{MFGs} and the underlying $N$-player games can be established in two directions: constructing approximate Nash equilibria for $N$-player games starting from a solution to the \mbox{MFG} (for instance, \citet{huangetal06}, \citet{carmonadelarue13}, \citet{gomesetal13}, to name a few), or by showing convergence of approximate $N$-player Nash equilibria to solutions of the \mbox{MFG}, as $N\to \infty$.
Crucial, especially in the second direction, is the choice of admissible strategies in the definition of $N$-player Nash equilibria. Particularly difficult is the question of convergence in the non-stationary case when Nash equilibria are considered in closed-loop strategies (Markov feedback strategies with full state information). A breakthrough in this direction was made in \citet{CDLL15}, where convergence of Nash equilibria is established through the so-called master equation provided the latter is well-posed, an assumption that  implies uniqueness of \mbox{MFG} solutions. 
More recently, in \cite{lacker18}, a general convergence result was proved in the non-degenerate diffusion setting, but to weak solutions of the \mbox{MFG}. For weak \mbox{MFG} solutions, the limiting flow of measures can be stochastic even without common noise. An important question, converse to the convergence result, is whether all weak \mbox{MFG} solutions can be obtained as limits of convergent closed-loop $N$-player approximate Nash equilibria. 
The analysis performed in \citet[Sect.~7]{lacker18} seems to suggest that this is not always possible. We believe that a way to have a characterization of all MFG solutions as limits of approximate Nash equilibria in $N$-player games is to consider a more general concept of solution, such as correlated equilibria.

\emph{Correlated equilibria} were introduced in two seminal papers by Robert Aumann \citep{aumann74,aumann87} for many-player games. Aumann's main idea can be explained as follows: a \emph{mediator} or \emph{correlation device} randomly selects a strategy profile according to some publicly known distribution, then recommends to each player \emph{in private} a strategy according to the profile. A probability distribution on the space of strategy profiles is a correlated equilibrium (\mbox{CE}, for short) if no player has an incentive to unilaterally deviate from the mediator's recommendation. In case the mediator uses a product probability distribution, we are back to a Nash equilibrium in mixed strategies. A classical concrete example of a mediator is that of a traffic light; see, for instance, Section 13.1.4 in \citep{roughgarden16} for more details. Moreover, we notice that the notion of \mbox{CE} admits other equivalent interpretations than that of a mediator. For instance, in \citep{barany92} it is shown that a \mbox{CE} of a non-cooperative $N$-person game ($N \ge 4$) coincides with a Nash equilibrium of an extended game where the players are allowed to communicate before the original game starts.

Originally introduced in the context of static games with complete information by Aumann, the new notion of \mbox{CE} gave rise to a huge literature in game theory as well as economics along many directions. We refer to the survey by \citet{forges12} on several aspects of the more general notion of communication equilibrium and extensions of \mbox{CE} to dynamic games, possibly stochastic and with incomplete information. More specifically on \mbox{CE} in stochastic games whose framework is close to ours, an extensive study has been performed in \citep{solan00,solan01,solanvieille02}.

The important role the concept of \mbox{CE} plays in game theory and economics can be explained by its many appealing properties, as compared to Nash equilibria. For instance, higher equilibrium payoffs can be reached, possibly outside the convex hull of Nash equilibrium payoffs. The computational complexity of \mbox{CE} is generally lower than for Nash equilibria \citep[see][]{gilboazemel89}. In the evolutionary game theory literature, it has been proved that if all players follow natural learning procedures then the empirical distribution of their actions converges to \mbox{CE} distributions \citep[for instance,][]{hart05}. Moreover, given their interpretation in terms of mediator's recommendations, correlated equilibria can be seen as intermediate configurations between the two extreme cases of decentralized solutions such as Nash equilibria on the one hand, and centrally planned optimal solutions that are forced on the players on the other hand.

Here, we consider correlated equilibria for a simple class of symmetric finite horizon $N$-player games and their natural \mbox{MFG} counterpart as $N \to \infty$. In the $N$-player setting, the state variables evolve in discrete time, both state space and the set of control actions are finite and, more importantly, the players are allowed to use only restricted strategies, that is, feedback strategies that depend only on time and the corresponding individual state variable. We believe that further extensions of our results to games in continuous time and with common noise are possible. However, they are postponed to future research as they will most probably require different techniques.

Within this framework, we propose a notion of \emph{correlated \mbox{MFG} solution, 
defined as a probability distribution over the space of all pairs of strategies and flows of measures such that (i) the ``representative player'' has no incentive to deviate from the mediator's recommendation; (ii) the flow of measures at any time $t$ equals the marginal law of the state variable at time $t$ conditioned on the $\sigma$-algebra generated by the whole flow of measures up to the terminal time. The main contribution of the paper consists in justifying the definition of correlated MFG solutions in the following two ways:} 
\begin{itemize}
	
	\item We prove that any sequence of symmetric approximate correlated equilibria in restricted strategies for the $N$-player games subsequentially converges towards some correlated \mbox{MFG} solution according to the definition above (forward approximation).
	
	\item We also prove the converse backward approximation, that is, any correlated \mbox{MFG} solution arises as limit of symmetric approximate \mbox{CE} in the $N$-player games as $N \to \infty$ or, in other terms, any correlated \mbox{MFG} solution can be implemented in a natural way by some mediator willing to recommend strategies to the players.
\end{itemize}
Both approximation results will be proved using a purely probabilistic approach heavily relying on the theory of weak convergence of probability measures as well as coupling arguments.
  
The rest of the paper is structured as follows. In Section~\ref{SectNotation}, we introduce the notation and basic elements for the objects of our study. In Section~\ref{SectPrelimit}, we describe the underlying $N$-player games and give the definition of (approximate) correlated equilibrium in restricted strategies. Moreover, we also prove that $N$-player correlated equilibria exist in the class of symmetric profile distributions. Section~\ref{SectMFG} is dedicated to the mean field limit model. There, we give our definition of correlated \mbox{MFG} solution. An example of a correlated \mbox{MFG} with explicit solutions is provided in Section~\ref{SectMFGExample}. In Section~\ref{SectConv}, we show that symmetric $N$-player correlated equilibria concentrate, in the limit as $N\to \infty$, on correlated \mbox{MFG} solutions, while Section~\ref{SectApprox} contains the converse result, that is, any correlated \mbox{MFG} solution arises as a limit of symmetric approximate $N$-player correlated equilibria as $N \to \infty$. In Appendix~\ref{AppAux}, we collect some auxiliary results.


\section{Preliminaries} \label{SectNotation}

For a Polish space $\mathcal{S}$, we denote by $\prbms{\mathcal{S}}$ the space of probability measures on $\Borel{\mathcal{S}}$, the Borel sets of $\mathcal{S}$, and endow $\prbms{\mathcal{S}}$ with the topology of weak convergence of measures. Many of the spaces of interest here are simply finite sets. We endow a finite set with the discrete topology, which makes it a Polish space (a compatible metric being the discrete metric).

If $\mathcal{S}$ is finite, then a metric on $\prbms{\mathcal{S}}$ compatible with the weak convergence topology is given by the following $L^{1}$-distance, which we indicate by $\mathsf{dist}$ when the underlying space is clear from the context:
\[
	\mathsf{dist}(m,\tilde{m})\doteq \frac{1}{2} \sum_{x\in \mathcal{S}} | m(x) - \tilde{m}(x)|,\quad m,\tilde{m}\in \prbms{\mathcal{S}}. 
\]
Notice that weak convergence and convergence in total variation coincide for probability measures over a finite set. If $m$, $\tilde{m}$ are empirical measures of the same size, that is, if $m = \frac{1}{N} \sum_{i=1}^{N} \delta_{s_{i}}$, $\tilde{m} = \frac{1}{N} \sum_{i=1}^{N} \delta_{\tilde{s}_{i}}$ for some $s_{i}, \tilde{s}_{i} \in \mathcal{S}$, $i\in \{1,\ldots,N\}$, then
\begin{equation} \label{EqEMdist}
	\mathsf{dist}(m,\tilde{m}) \leq \min_{\sigma \text{ permutation of } \{1,\ldots,N\}} \frac{1}{N} \sum_{i=1}^{N} \mathbf{1}_{s_{i} \neq \tilde{s}_{\sigma(i)}} \\
	\leq \frac{1}{N} \sum_{i=1}^{N} \mathbf{1}_{s_{i} \neq \tilde{s}_{i}}.
\end{equation}

We consider symmetric dynamic games in discrete time over a finite time horizon with individual state and action spaces given by finite sets. Admissible strategies will have Markov feedback form but with information restricted to player's individual states (sometimes called ``Markov open-loop''). To fix the notation, we choose
\begin{itemize}
	\item $T\in \mathbb{N}$, representing the finite time horizon (with initial time zero);
	
	\item non-empty finite sets $\mathcal{X}$ and $\Gamma$, the set of individual states and control actions, respectively;
	
	\item a measurable function $\Psi\!: \{0,\ldots,T\!-\!1\}\times \mathcal{X}\times \prbms{\mathcal{X}}\times \Gamma\times \mathcal{Z}  \rightarrow \mathcal{X}$, the system function, determining the one-step individual state dynamics, where $\mathcal{Z}\doteq [0,1]$ is the space of noise states;
	
	\item a bounded measurable function $f\!: \{0,\ldots,T\!-\!1\}\times \mathcal{X}\times \prbms{\mathcal{X}}\times \Gamma \rightarrow \mathbb{R}$, representing the running costs;
	
	\item a bounded measurable function $F\!: \mathcal{X}\times \prbms{\mathcal{X}} \rightarrow \mathbb{R}$, representing the terminal costs.
	
\end{itemize}

Denote by $\nu$ the uniform distribution on the Borel sets of $\mathcal{Z} = [0,1]$; $\nu$ will be the common distribution of the random variables representing the idio\-syncratic noise.

Let $\mathcal{R}$ denote the set of Markov feedback strategies over players' own states (\emph{restricted strategies}):
\[
	\mathcal{R} \doteq \left\{ \phi\!: \{0,\ldots,T\!-\!1\}\times \mathcal{X}\rightarrow \Gamma \right\}.
\]
Notice that $\mathcal{R}$ is a finite set; it will hence be endowed with the discrete topology. Let $\mathcal{U}$ denote the set of mappings from $\mathcal{R}$ to $\mathcal{R}$:
\[
	\mathcal{U}\doteq \left\{ u\!: \mathcal{R} \rightarrow \mathcal{R} \right\}.
\]
Since $\mathcal{R}$ is a finite set, $\mathcal{U}$ is finite, too; it will therefore be endowed with the discrete topology. Any element of $\mathcal{U}$, that is, any function $u\!: \mathcal{R}\rightarrow \mathcal{R}$ (which is automatically measurable) will be referred to as a \emph{strategy modification}.

For the $N$-player game, we have to consider probability measures on strategy vectors (or strategy profiles). Any such probability measure, that is, any element of $\prbms{\mathcal{R}^{N}}$, will be called a \emph{correlated profile}. For the mean field game, we will consider probability measures on individual strategies times flows of state distributions. Any such probability measure, that is, any element of $\prbms{\mathcal{R} \times \prbms{\mathcal{X}}^{T+1}}$, will be called a \emph{correlated flow}.


\section{The $N$-player games} \label{SectPrelimit}

Fix $N\in \mathbb{N}$. Choose $\mathfrak{m}^{N} \in \prbms{\mathcal{X}^{N}}$, the joint distribution of the players' states at time zero; for instance, $\mathfrak{m}^{N} = \otimes^N \mathfrak{m}_{0}$ for some $\mathfrak m_0 \in \mathcal P(\mathcal X)$. Let $\gamma^{N} \in \prbms{\mathcal{R}^{N}}$ be an $N$-player correlated profile, and let $u\in \mathcal{U}$ be a strategy modification.

A tuple $((\Omega,\mathcal{F},\Prb),\Phi^{N}_1,\ldots,\Phi^{N}_N,X^{N}_{1}(.),\ldots,X^{N}_{N}(.),\xi^{N}_{1}(.),\ldots,\xi^{N}_{N}(.))$ is called a \emph{realization} of the triple $(\mathfrak{m}^{N},\gamma^{N},u)$ \emph{for player $i$} if $\Phi^{N}_1,\ldots,\Phi^{N}_N$ are $\mathcal{R}$-valued random variables, $X^{N}_{j}(t)$, $j\in \{1,\ldots,N\}$, $t\in \{0,\ldots,T\}$, $\mathcal{X}$-valued random variables, and $\xi^{N}_{j}(t)$, $j\in \{1,\ldots,N\}$, $t\in \{1,\ldots,T\}$, are $\mathcal{Z}$-valued random variables all defined on the probability space $(\Omega,\mathcal{F},\Prb)$ such that 
\begin{enumerate}[(i)]

	\item $\Prb\circ (X^{N}_{1}(0),\ldots,X^{N}_{N}(0))^{-1} = \mathfrak{m}^{N}$;
	
	\item $\Prb\circ (\Phi^{N}_1,\ldots,\Phi^{N}_N)^{-1} = \gamma^{N}$;

	\item $\xi^{N}_{j}(t)$, $j\in \{1,\ldots,N\}$, $t\in \{1,\ldots,T\}$, are independent and identically distributed (i.i.d.) with common distribution $\Prb\circ (\xi^{N}_{j}(t))^{-1} = \nu$;

	\item\label{PrelimitRealizationIndependence} $(\xi^{N}_{j}(t))_{j\in \{1,\ldots,N\}, t\in \{1,\ldots,T\} }$, $(X^{N}_{j}(0))_{j\in \{1,\ldots,N\}}$, and $(\Phi^{N}_{j})_{j\in \{1,\ldots,N\}}$ are independent as random variables with values in $\mathcal{Z}^{N\cdot T}$, $\mathcal{X}^{N}$, and $\mathcal{R}^{N}$, respectively;

	\item\label{PrelimitDynamics} $\Prb$-almost surely, for every $t\in \{0,\ldots,T-1\}$,
	\begin{align*}
	X^{N}_{i}(t+1) &= \Psi\left(t,X^{N}_{i}(t),\mu^{N}_{i}(t),u\circ\Phi^{N}_{i}\left(t,X^{N}_{i}(t)\right), \xi^{N}_{i}(t+1) \right),\\
	X^{N}_{j}(t+1) &= \Psi\left(t,X^{N}_{j}(t),\mu^{N}_{j}(t),\Phi^{N}_{j}\left(t,X^{N}_{j}(t)\right), \xi^{N}_{j}(t+1) \right),\quad j\neq i,
	\end{align*}
	where $\mu^{N}_{l}(t)$ is the empirical measure of the states at time $t$ of all players except player $l$:
	\[
		\mu^{N}_{l}(t) \doteq \frac{1}{N-1} \sum_{j\neq l} \delta_{X^{N}_{j}(t)},\quad l\in \{1,\ldots,N\}.
	\]

\end{enumerate}
The difference in the dynamics of $X^N_i$ and $X^N _j$, $j\neq i$, above is that the former includes the strategy modification $u$. Any realization 
\[((\Omega,\mathcal{F},\Prb),\Phi^{N}_1,\ldots,\Phi^{N}_N,X^{N}_{1}(.),\ldots,X^{N}_{N}(.),\xi^{N}_{1}(.),\ldots,\xi^{N}_{N}(.))\] 
of the triple $(\mathfrak{m}^{N},\gamma^{N},u)$ for player $i$ can be interpreted in the following way. The random variables $X^{N}_{1}(0),\ldots,X^{N}_{N}(0)$ represent the initial states of players $1$ through $N$, and their joint distribution is given by $\mathfrak{m}^{N}$. The random variable $\Phi^{N}_l$ represents the recommendation (or signal) the mediator sends to player $l$ before the game starts. While the joint distribution of $\Phi^{N}_{1},\ldots,\Phi^{N}_N$, which is equal to $\gamma^{N}$, is common knowledge, no player can directly see the recommendations received by the others. This feature is made precise in the way the state dynamics are formulated in \eqref{PrelimitDynamics}: Player $i$, the player who might deviate, chooses a strategy by modifying the recommendation $\Phi^{N}_{i}$ through the application of a mapping $u\!: \mathcal{R}\rightarrow \mathcal{R}$, while the other players follow the recommendation they receive from the mediator. Player $i$ thus uses the random strategy $u\circ \Phi^{N}_{i}$ instead of simply $\Phi^{N}_{i}$. Clearly, $u\circ \Phi^{N}_{i}$ is $\sigma(\Phi^{N}_{i})$-measurable.

\begin{rem}
The independence assumption in \eqref{PrelimitRealizationIndependence} is crucial. Clearly, the vector $(\xi^{N}_{j}(t))$ of noise variables and the vector $(X^{N}_{j}(0))$ of initial states have to be independent. But we also require them to be independent of the vector $(\Phi^{N}_{j})$ of recommendation variables. This makes precise the idea that the mediator gives recommendations to the players before the game starts. Recall that what the mediator suggests are feedback strategies. If player $j$ accepts the mediator's recommendation, then, given a scenario $\omega\in \Omega$, he will use the feedback strategy $\Phi^{N}_{j}(\omega)$. In view of the dynamics according to \eqref{PrelimitDynamics}, he will therefore select, at any time $t$, the control action $\Phi^{N}_{j}(\omega)(t,X^{N}_{j}(t,\omega))$. The control action at time $t$ is thus in general not independent of the noise variables up to time $t$, nor is it independent of the initial states. An analogous observation holds for player $i$, who might modify the mediator's recommendation.
\end{rem}

The costs for player $i$ associated with initial distribution $\mathfrak{m}^{N}$, correlated profile $\gamma^{N}$, and a strategy modification $u$ are given by
\[
\begin{split}
	J^{N}_{i}\bigl( \mathfrak{m}^{N},\gamma^{N},u \bigr) \doteq \Mean\Biggl[ \sum_{t=0}^{T-1} f\left(t,X^{N}_{i}(t),\mu^{N}_{i}(t),u\circ \Phi^{N}_{i}\left(t,X^{N}_{i}(t)\right) \right) + F\left(X^{N}_{i}(T),\mu^{N}_{i}(T) \right) \Biggr],
\end{split}
\]
where the expected value on the right-hand side above is computed with respect to any realization of the triple $(\mathfrak{m}^{N},\gamma^{N},u)$ for player $i$. Thanks to the independence assumption \eqref{PrelimitRealizationIndependence}, the costs are well defined in that they do not depend on the choice of the realization.

\begin{defn} \label{DefPrelimitCE}
	Let $\epsilon \ge 0$. A correlated profile $\gamma^{N} \in \prbms{\mathcal{R}^N}$ is called an $\epsilon$-\emph{correlated equilibrium} in restricted strategies with initial distribution $\mathfrak{m}^N$ if for every $i\in \{1,\ldots,N\}$, every strategy modification $u \in \mathcal{U}$,
	\[
		J^{N}_{i}(\mathfrak{m}^{N},\gamma^{N},\Id) \leq J^{N}_{i}(\mathfrak{m}^{N},\gamma^{N},u) + \epsilon. 
	\] 
When $\epsilon = 0$, we say that $\gamma^{N}$ is a correlated equilibrium in restricted strategies.
\end{defn}

An $\epsilon$-correlated equilibrium is called \emph{symmetric} if it is symmetric as a probability measure on $\mathcal{R}^N$ (i.e.\ invariant under permutations of the components).

\begin{rem}
Nash equilibria are particular cases of correlated equilibria. According to Definition~\ref{DefPrelimitCE}, a Nash equilibrium in mixed strategies corresponds to a correlated profile $\gamma^{N}$ that has product form, while a Nash equilibrium in pure strategies corresponds to a correlated profile which is the product of Dirac measures concentrated in the strategies of the Nash profile. 
\end{rem}

Next, we prove that there always exists a symmetric correlated equilibrium for the $N$-player game. To this end, instead of relying on the existence of symmetric Nash equilibria, which would hold in this setting, we rather follow a more direct approach by applying the existence result in \citet{HS} through a simple symmetrization argument. 

\begin{prop} \label{PropExistence}
	Let $\mathfrak{m}^{N} \in \prbms{\mathcal{X}^{N}}$ be symmetric. Then there exists a symmetric correlated equilibrium with initial distribution $\mathfrak{m}^{N}$.
\end{prop}

\begin{proof}
Applying \citet[Theorem 1]{HS} to our setting, we obtain the existence of a correlated equilibrium $\gamma \in \mathcal P(\mathcal R^N)$ for the $N$-player game described above, i.e.\ for all $i \in \{1,\ldots,N\}$ and all $u \in \mathcal U$ we have
\begin{equation}\label{payoff-aux}
	\sum_{\varphi\in \mathcal R^N} \gamma(\varphi) \left(J^{N}_i (\mathfrak{m}^N,\delta_{\varphi} , u) - J^{N}_i (\mathfrak{m}^N,\delta_{\varphi}, \Id)\right) \ge 0.
\end{equation}
Since $\gamma$ is not necessarily symmetric, we symmetrize it by defining
\[ \tilde \gamma (\varphi^1, \ldots, \varphi^N) := \frac{1}{N!} \sum_\sigma \gamma(\varphi^{\sigma(1)},\ldots, \varphi^{(N)}),\]
where $\sigma$ varies in the set of all permutations of $\{1,\ldots,N\}$. We check that also $\tilde \gamma$ is a CE for the $N$-player game. Letting $\varphi^\sigma = (\varphi^{\sigma(1)},\ldots,\varphi^{\sigma(N)})$, we can write
\begin{eqnarray*}
	\sum_{\varphi\in \mathcal R^N} \tilde \gamma(\varphi) \left(J^{N}_i (\mathfrak{m}^N,\delta_{\varphi} , u) - J^{N}_i (\mathfrak{m}^N,\delta_{\varphi}, \Id)\right)\\
	= \frac{1}{N!} \sum_\sigma \sum_{\varphi\in \mathcal R^N} \gamma(\varphi^{\sigma(1)},\ldots,\varphi^{\sigma(N)}) \left(J^{N}_i (\mathfrak{m}^N,\delta_{\varphi} , u) - J^{N}_i (\mathfrak{m}^N,\delta_{\varphi}, \Id)\right)\\
	= \frac{1}{N!} \sum_\sigma \sum_{\varphi\in \mathcal R^N} \gamma(\varphi^\sigma ) \left(J^{N}_{\sigma(i)} (\mathfrak{m}^N,\delta_{\varphi^\sigma} , u) - J^{N}_{\sigma(i)} (\mathfrak{m}^N,\delta_{\varphi^\sigma}, \Id)\right) \ge 0,
\end{eqnarray*}
where the second equality is due to symmetry and the final inequality follows from \eqref{payoff-aux}.
\end{proof}


\section{The correlated mean field game} \label{SectMFG}

Choose $\mathfrak{m}_{0} \in \prbms{\mathcal{X}}$, the distribution of the representative player's state at time zero. Let $\rho\in \prbms{\mathcal{R} \times \prbms{\mathcal{X}}^{T+1}}$ be a correlated flow, and let $u\in \mathcal{U}$ be a strategy modification.

A tuple $((\Omega,\mathcal{F},\Prb),\Phi,X(.),\mu(.),\xi(.))$ is called a \emph{realization} of the triple $(\mathfrak{m}_{0},\rho,u)$ if $\Phi$ is an $\mathcal{R}$-valued random variable, $X(0),\ldots,X(T)$ are $\mathcal{X}$-valued random variables, $\mu(0),\ldots,\mu(T)$ are $\prbms{\mathcal{X}}$-valued random variables, and $\xi(1),\ldots,\xi(T)$ are $\mathcal{Z}$-valued random variables all defined on a common probability space $(\Omega,\mathcal{F},\Prb)$ such that
\begin{enumerate}[(i)]
	
	\item $\Prb\circ (X(0))^{-1} = \mathfrak{m}_{0}$;
		
	\item\label{LimitRealizationCorrelFlow} $\Prb\circ (\Phi,\mu(0),\ldots,\mu(T))^{-1} = \rho$;
	
	\item $\xi(t)$, $t\in \{1,\ldots,T\}$, are i.i.d.\ with common distribution $\Prb\circ (\xi(t))^{-1} = \nu$;
	
	\item\label{LimitRealizationIndependence} $\xi(.)$, $X(0)$, and $(\Phi, \mu(.))$ are independent as random variables with values in $\mathcal{Z}^{T}$, $\mathcal{X}$, and $\mathcal{R} \times \prbms{\mathcal{X}}^{T+1}$, respectively;

	\item $\Prb$-almost surely, for every $t\in \{0,\ldots,T-1\}$,
	\begin{equation} \label{EqLimitDyn}
		X(t+1) = \Psi\left(t,X (t),\mu(t),u\circ\Phi\left(t,X (t)\right), \xi(t+1) \right).
	\end{equation}

\end{enumerate}

Recalling the heuristic connection between $N$-player games and mean field game, we can interpret a realization $((\Omega,\mathcal{F},\Prb),\Phi,X(.),\mu(.),\xi(.))$ of the triple $(\mathfrak{m}_{0},\rho,u)$ as follows. The random variable $\Phi$ represents the recommendation that one representative player receives from the mediator, whereas $X(.)$ gives the representative player's state sequence, which is recursively determined through Eq.~\eqref{EqLimitDyn}. There, $\xi(t)$, $t\in \{1,\ldots,T\}$, are the noise variables, while $\mu(.)$ represents a stochastic flow of measures. 

\begin{rem}
	The flow of measures $\mu(.)$ should be thought of as a limit point of the $N$-player flows of empirical measures. As such, it will in general be stochastic and not independent of the recommendation variable $\Phi$, which in turn should be thought of as a limit point of the recommendation variables for one fixed player, say the first, in the $N$-player games. It is therefore necessary to prescribe the joint distribution of $\Phi$ and $\mu$, as done in \eqref{LimitRealizationCorrelFlow} through the correlated flow $\rho$. Similarly, in \eqref{LimitRealizationIndependence}, which should be compared to the independence assumption \eqref{PrelimitRealizationIndependence} of the $N$-player games, we require independence of $\xi(.)$, $X(0)$, and $(\Phi, \mu(.))$, not just of $\xi(.)$, $X(0)$, and $\Phi$. We stress that in general $\Phi$ and $\mu$ will not be independent. 
\end{rem}

The costs for a representative player associated with initial distribution $\mathfrak{m}_{0}$, correlated flow $\rho$, and a strategy modification $u\!: \mathcal{R} \rightarrow \mathcal{R}$ are given by
\[
\begin{split}
	J(\mathfrak{m}_{0},\rho,u) \doteq \Mean\Biggl[ \sum_{t=0}^{T-1} f\left(t,X (t),\mu(t),u\circ \Phi\left(t,X (t)\right) \right) + F\left(X(T),\mu(T) \right) \Biggr],
\end{split}
\]
where the expected value on the right-hand side above is computed with respect to any realization of $(\mathfrak{m}_{0},\rho,u)$. Thanks to \eqref{LimitRealizationIndependence}, any two realizations of $(\mathfrak{m}_{0},\rho,u)$ generate the same expected value. The cost functional $J$ is thus well defined.

\begin{defn} \label{DefCE}
	A correlated flow $\rho \in \prbms{\mathcal{R} \times \prbms{\mathcal{X}}^{T+1}}$ is called a \emph{correlated solution of the mean field game} in restricted strategies with initial distribution $\mathfrak{m}_{0}$ if the following two conditions hold:
\begin{enumerate}[(i)]
	\item Optimality: For every strategy modification $u \in \mathcal{U}$,
	\[
		J(\mathfrak{m}_{0};\rho,\Id) \leq J(\mathfrak{m}_{0};\rho,u).
	\]

	\item Consistency: If $((\Omega,\mathcal{F},\Prb),\Phi,X(.),\mu(.),\xi(.))$ is a realization of the triple $(\mathfrak{m}_{0},\rho,\Id)$, then for every $t\in \{0,\ldots,T\}$,
	\[
		\mu(t) (.) = \Prb\left( X(t)\in . \;|\; \mathcal{F}^{\mu} \right),
	\]
	where $\mathcal{F}^{\mu} \doteq \sigma(\mu) = \sigma(\mu(s):s \in \{0,\ldots, T\})$.
\end{enumerate}
\end{defn}

The consistency condition in Definition~\ref{DefCE} is to be understood in the sense that $\mu(t)$ is a regular conditional distribution of $X(t)$ given  $\mathcal{F}^{\mu}$. Notice that $\mathcal{F}^{\mu}$ is the $\sigma$-algebra generated by the entire flow of measures $\mu$, up to terminal time $T$. Intuitively, the consistency condition can be interpreted as follows: The moderator has an idea of the flow $\mu$ on the whole time interval, on the basis of which he recommends strategies to the players. If each player follows his recommendations, then that flow $\mu$ will arise from aggregation of the individual behaviors.

\begin{rem}
Definition~\ref{DefCE} should be compared to the definition of weak \mbox{MFG} solution, more precisely \emph{weak semi-Markov mean field equilibrium}, given
in \citet[][Definition~2.5]{lacker18}. An obvious difference lies in the dynamics: While Lacker works with controlled It{\^o} diffusions driven by non-degenerate additive Wiener noise (as here, idiosyncratic, no common noise), here we consider simple discrete time dynamics with finite state and control space. Conceptually more important is the fact that the admissible strategies here are restricted to functions that depend only on time and the player's current state, while Lacker allows for an additional dependence on the flow of measures up to current time. Notice that the flow of measures may be stochastic in both cases. Clearly, there is no mediator or correlation device in \citet{lacker18}. If in our situation we take the recommendation variable to be (almost surely) constant, hence with Dirac distribution concentrated at some feedback strategy $\phi \in \mathcal{R}$, then the optimality condition in Definition~\ref{DefCE} above can be seen to be analogous to the optimality condition in Lacker's definition (point (5) there). His consistency condition (point (6) there) is apparently different in that the conditional distribution is taken with respect to the $\sigma$-algebra generated by the flow of measures up to current time, not up to terminal time as in our definition. However, if the recommendation variable is (almost surely) constant, or absent as in Lacker's work, then the two ways of conditioning lead to equivalent consistency conditions, thanks to the (semi-)Markov property of the state process. Indeed, when $\Phi$ is constant, $X(t)$ and $\mu$ are conditionally independent given $\mathcal F^\mu (t) = \sigma(\mu(s): s\in \{0,\ldots,t\})$, for all $t\in \{0,\ldots, T\}$. Therefore, in this case, the property $\mu(t) (.) = \mathbf P(X(t) \in . \; |\; \mathcal F^\mu (t))$, $t\in \{0,\ldots, T\}$, implies the consistency condition (ii) above. In this way, and under the simplifying assumptions made here, one can interpret weak \mbox{MFG} solutions as a special case of correlated solutions.
\end{rem}

\section{Example of a correlated mean field game} \label{SectMFGExample}

In this section, we give an example of a two-state mean field game possessing correlated solutions with non-deterministic flow of measures.

Set $\mathcal{X} \doteq \{-1, 1\}$, $\Gamma \doteq \{0,1\}$, and $T \doteq 2$. Define the system function $\Psi$ by
\[
	\Psi(t,x,m,\gamma,z) \doteq \Psi(x,\gamma,z) \doteq \begin{cases}
		x &\text{if } \gamma = 0 \text{ and } z\in [0,\frac{1}{2}], \\
		-x &\text{if } \gamma = 0 \text{ and } z\in (\frac{1}{2},1], \\
		x &\text{if } \gamma = 1 \text{ and } z\in [0,\frac{3}{4}], \\
		-x &\text{if } \gamma = 1 \text{ and } z\in (\frac{3}{4},1],
	\end{cases}
\]
Notice that $\Psi$ is time-homogeneous and independent of the measure variable, which justifies our slight abuse of notation. According to $\Psi$, when moving one step in time, the player's state switches with probability $1/2$ if action $\gamma = 0$ is chosen, while it will change only with probability $1/4$ if action $\gamma = 1$ is played. Choose running costs $f$ and terminal costs $F$ according to
\begin{align*}
	& f(t,x,m,\gamma) \doteq \begin{cases}
	c_{0}\cdot \gamma &\text{if } t = 0, \\
	c_{1}\cdot \gamma -x\cdot \mathbf{M}(m) &\text{if } t = 1,
	\end{cases} & 
	& F(x,m) \doteq -x\cdot \mathbf{M}(m), &
\end{align*}
where $c_{0}, c_{1} > 0$ will be chosen below and $\mathbf{M}(m) \doteq m(\{1\}) - m(\{-1\})$ is the mean of a probability measure $m \in \prbms{\{-1, 1\}}$.

Define strategies $\phi_{+}, \hat{\phi}_{+}, \phi_{-}, \hat{\phi}_{-}, \phi_{o} \in \mathcal{R}$ according to
\begin{align*}
& \phi_{+}(t,x) \doteq \begin{cases}
	0 &\text{if } x = -1 \\
	1 &\text{if } x = 1,
\end{cases} &
& \hat{\phi}_{+}(t,x) \doteq \begin{cases}
0 &\text{if } x = -1 \text{ or } t = 1, \\
1 &\text{if } x = 1 \text{ and } t = 0,
\end{cases} & && \\
& \phi_{-}(t,x) \doteq \begin{cases}
0 &\text{if } x = 1, \\
1 &\text{if } x = -1,
\end{cases} &
& \hat{\phi}_{-}(t,x) \doteq \begin{cases}
0 &\text{if } x = 1 \text{ or } t = 1, \\
1 &\text{if } x = -1 \text{ and } t = 0,
\end{cases} & \\
& \phi_{o}(t,x) \doteq 0, & & t\in \{0,1\},\; x\in \{-1,1\}. &
\end{align*}
Strategy $\phi_{+}$ has the effect of maximizing the probability of being in state $1$ at times $1$ and $2$, while $\hat{\phi}_{+}$ only maximizes that probability at time $1$. The effect of $\phi_{-}$, $\hat{\phi}_{-}$ is analogous, with the roles of states $1$ and $-1$ inverted. Under strategy $\phi_{o}$ (``do nothing''), on the other hand, the two states will have equal probability at both time $1$ and time $2$, independently of the initial state.

Choose $\beta_{1},\dots,\beta_{4} > 0$ such that
\begin{align} \label{EqExmplBeta}
	& \beta_{1} + \beta_{2} + \beta_{3} + \beta_{4} = \frac{1}{2}, & & \frac{5\beta_{1} + 4\beta_{2}}{8(\beta_{1}+\beta_{2})} = \frac{5\beta_{3} + 4\beta_{4}}{8(\beta_{3}+\beta_{4})}. & 
\end{align}
An obvious choice satisfying \eqref{EqExmplBeta} is to set $\beta_{i} \doteq 1/8$ for all $i\in \{1,\ldots,4\}$. Define distributions in $\prbms{\{-1, 1\}}$ according to
\begin{align*}
	& m_{0} \doteq \frac{1}{2}\delta_{1} + \frac{1}{2}\delta_{-1}, & && \\
	& m_{1}^{+} \doteq \frac{5\beta_{1} + 4\beta_{2}}{8(\beta_{1} + \beta_{2})} \delta_{1} + \frac{3\beta_{1} + 4\beta_{2}}{8(\beta_{1} + \beta_{2})} \delta_{-1}, &
	& m_{2}^{+} \doteq \frac{21\beta_{1} + 16\beta_{2}}{32(\beta_{1} + \beta_{2})} \delta_{1} + \frac{11\beta_{1} + 16\beta_{2}}{32(\beta_{1} + \beta_{2})} \delta_{-1}, & \\
	& m_{1}^{-} \doteq \frac{3\beta_{1} + 4\beta_{2}}{8(\beta_{1} + \beta_{2})} \delta_{1} + \frac{5\beta_{1} + 4\beta_{2}}{8(\beta_{1} + \beta_{2})} \delta_{-1}, &
	& m_{2}^{-} \doteq \frac{11\beta_{1} + 16\beta_{2}}{32(\beta_{1} + \beta_{2})} \delta_{1} + \frac{21\beta_{1} + 16\beta_{2}}{32(\beta_{1} + \beta_{2})}\delta_{-1}. &
\end{align*}

In terms of the above distributions and strategies, define a correlated flow $\rho$ by
\[
\begin{split}
	\rho &\doteq 
	\beta_{1} \delta_{(\phi_{+},(m_{0},m_{1}^{+},m_{2}^{+}))}
	+ \beta_{2} \delta_{(\phi_{o},(m_{0},m_{1}^{+},m_{2}^{+}))}
	+ \beta_{3} \delta_{(\hat{\phi}_{+},(m_{0},m_{1}^{+},m_{0}))}
	+ \beta_{4} \delta_{(\phi_{o},(m_{0},m_{1}^{+},m_{0}))} \\
	&+ \beta_{1} \delta_{(\phi_{-},(m_{0},m_{1}^{-},m_{2}^{-}))}
	+ \beta_{2} \delta_{(\phi_{o},(m_{0},m_{1}^{-},m_{2}^{-}))}
	+ \beta_{3} \delta_{(\hat{\phi}_{-},(m_{0},m_{1}^{-},m_{0}))}
	+ \beta_{4} \delta_{(\phi_{o},(m_{0},m_{1}^{-},m_{0}))}.
\end{split}
\]

In Figure~\ref{FigTrajectories}, we illustrate the four measure trajectories that have strictly positive probability according to $\rho$, namely $(m_{0},m_{1}^{+},m_{2}^{+})$, $(m_{0},m_{1}^{+},m_{0})$, $(m_{0},m_{1}^{-},m_{2}^{-})$, and $(m_{0},m_{1}^{-},m_{0})$. In analogy with the classical definition of correlated equilibria, we suppose that $\rho$, which here gives the joint distribution of the mediator's recommendations and the flow of measures, is common knowledge. Thus, if the representative player receives the recommendation to play $\phi_{+}$, she can infer that with probability one the flow of measures will be concentrated at the measure trajectory $(m_{0},m_{1}^{+},m_{2}^{+})$. Similarly, upon receiving recommendation $\hat{\phi}_{+}$, she can deduce that the flow of measures will be concentrated at $(m_{0},m_{1}^{+},m_{0})$, and analogously for $\phi_{-}$, $\hat{\phi}_{-}$. If, on the other hand, the mediator recommends to play  $\phi_{o}$, then all four measure trajectories are possible. In fact, $(m_{0},m_{1}^{+},m_{2}^{+})$ and $(m_{0},m_{1}^{-},m_{2}^{-})$ will both have (conditional) probability equal to $\frac{\beta_{2}}{2(\beta_{2}+\beta_{4})}$, while $(m_{0},m_{1}^{+},m_{0})$ and $(m_{0},m_{1}^{-},m_{0})$ will both have (conditional) probability equal to $\frac{\beta_{4}}{2(\beta_{2}+\beta_{4})}$; also see Remark~\ref{RemCorrelatedExmpl} below.

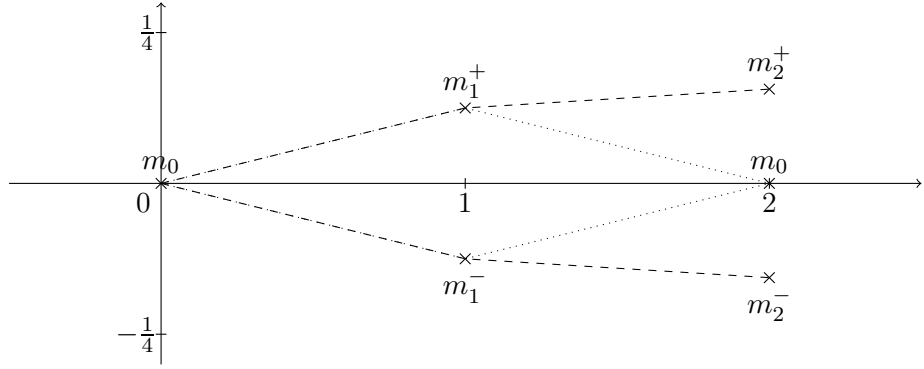
\begin{figure}
\begin{center}
\begin{tikzpicture}[xscale=4, yscale=8]

\tikzset{
	crocetta/.pic=
	{
		\draw [rotate=45] (0,-0.1) -- (0,0.1);
		\draw [rotate=-45] (0,-0.1) -- (0,0.1);
	}
}

	\draw[->] (-0.5,0) -- (2.5,0);
	\draw[->] (0,-0.3) -- (0,0.3);
	
	\foreach \x in {1,2}
		\draw (\x,-0.25pt) -- (\x,0.25pt);

	\draw (0,0) node[anchor=north east] {0};
	\foreach \x in {1,2}
		\draw (\x,0) node[anchor=north] {\x};

	\draw (-0.5pt,0.25) -- (0.5pt,0.25) node[anchor=east] {$\frac{1}{4}$};
	\draw (-0.5pt,-0.25) -- (0.5pt,-0.25) node[anchor=east] {$-\frac{1}{4}$};

	\draw[dashed] (0,0) -- (1,0.125) node[anchor=south] {$m^{+}_{1}$} -- (2,0.15625) node[anchor=south] {$m^{+}_{2}$};
	
	\draw[dotted] (0,0) node[anchor=south] {$m_{0}$} -- (1,0.125) -- (2,0) node[anchor=south] {$m_{0}$};
	
	\draw[dashed] (0,0) -- (1,-0.125) node[anchor=north] {$m^{-}_{1}$} -- (2,-0.15625) node[anchor=north] {$m^{-}_{2}$};
	
	\draw[dotted] (0,0) -- (1,-0.125) -- (2,0);
	
	\draw (0,0) pic {crocetta};
	\draw (1,0.125) pic {crocetta};
	\draw (1,-0.125) pic {crocetta};
	\draw (2,0.15625) pic {crocetta};
	\draw (2,-0.15625) pic {crocetta};
	\draw (2,0) pic {crocetta};
	
\end{tikzpicture}

\caption{Measure trajectories with strictly positive probability according to the correlated flow $\rho$. Elements of $\mathcal{P}(\{-1,1\})$ are identified with their mean, measured along the vertical axis. Values have been computed for the case $\beta_{1} = \beta_{2} = \beta_{3} = \beta_{4} = 1/8$. Time is measured along the horizontal axis. Only times $0$, $1$, $2$ are relevant for the model, though for the sake of illustration trajectories are represented as if time were continuous.}\label{FigTrajectories}
\end{center}
\end{figure}

We are going to show that $\rho$ is a correlated solution of the mean field game in the sense of Definition~\ref{DefCE} provided $c_{0}, c_{1} > 0$ are taken sufficiently small. To this end, let $((\Omega,\mathcal{F},\Prb),\Phi,X(.),\mu(.),\xi(.))$ be a realization of the triple $(\mathfrak{m}_{0},\rho,\Id)$. Then, for every $t\in \{0,1,2\}$, $\Prb$-almost surely,
\[
	\Prb\left( X(t)\in \cdot \;|\; \mathcal{F}^{\mu} \right) = \begin{cases}
	\frac{\beta_{1}}{\beta_{1} + \beta_{2}}\Prb\circ (X_{+}(t))^{-1} + \frac{\beta_{2}}{\beta_{1} + \beta_{2}}\Prb\circ (X_{o}(t))^{-1} &\text{if } \mu = (m_{0},m_{1}^{+},m_{2}^{+}), \\
	\frac{\beta_{3}}{\beta_{3} + \beta_{4}}\Prb\circ (\hat{X}_{+}(t))^{-1} + \frac{\beta_{4}}{\beta_{3} + \beta_{4}}\Prb\circ (X_{o}(t))^{-1} &\text{if } \mu = (m_{0},m_{1}^{+},m_{0}), \\
	\frac{\beta_{1}}{\beta_{1} + \beta_{2}}\Prb\circ (X_{-}(t))^{-1} + \frac{\beta_{2}}{\beta_{1} + \beta_{2}}\Prb\circ (X_{o}(t))^{-1} &\text{if } \mu = (m_{0},m_{1}^{-},m_{2}^{-}), \\
	\frac{\beta_{3}}{\beta_{3} + \beta_{4}}\Prb\circ (\hat{X}_{-}(t))^{-1} + \frac{\beta_{4}}{\beta_{3} + \beta_{4}}\Prb\circ (X_{o}(t))^{-1} &\text{if } \mu = (m_{0},m_{1}^{-},m_{0}),
	\end{cases}
\]
where $X_{o}$, $X_{+}$, $\hat{X}_{+}$, $X_{-}$, $\hat{X}_{-}$ are the state processes that result from applying feedback strategies $\phi_{o}$, $\phi_{+}$, $\hat{\phi}_{+}$, $\phi_{-}$, and $\hat{\phi}_{-}$, respectively, with initial distribution $m_{0}$. Notice that these processes can be recursively defined through $\Psi$ on the given probability space $(\Omega,\mathcal{F},\Prb)$ in terms of the noise variables $\xi(1)$, $\xi(2)$, and the initial state $X(0)$. For instance, $X_{+}(.)$ is recursively determined by setting
\begin{align*} 
	& X_{+}(0) \doteq X(0), &  & X_{+}(t+1) = \Psi\left(X_{+} (t), \phi_{+}\left(t,X_{+} (t)\right), \xi(t+1) \right), &  & t \in \{0,1\}, & 
\end{align*}
where $\Psi$ is seen as a function of state, control, and noise only, according to its definition above. Recalling that $m_{0} = \frac{1}{2}\delta_{1} + \frac{1}{2}\delta_{-1}$, we find:
\begin{align*}
	& \Prb\circ (X_{o}(t))^{-1} = m_{0}, & &t\in \{0,1,2\}, & && \\
	& \Prb\circ (X_{+}(0))^{-1} = m_{0}, &
	& \Prb\circ (X_{+}(1))^{-1} = \frac{5}{8}\delta_{1} + \frac{3}{8}\delta_{-1}, &
	& \Prb\circ (X_{+}(2))^{-1} = \frac{21}{32}\delta_{1} + \frac{11}{32}\delta_{-1}, & \\
	& \Prb\circ (\hat{X}_{+}(0))^{-1} = m_{0}, &
	& \Prb\circ (\hat{X}_{+}(1))^{-1} = \frac{5}{8}\delta_{1} + \frac{3}{8}\delta_{-1}, & & \Prb\circ (\hat{X}_{+}(2))^{-1} = m_{0}, & \\
	& \Prb\circ (X_{-}(0))^{-1} = m_{0}, & & \Prb\circ (X_{-}(1))^{-1} = \frac{3}{8}\delta_{1} + \frac{5}{8}\delta_{-1}, & & \Prb\circ (X_{-}(2))^{-1} = \frac{11}{32}\delta_{1} + \frac{21}{32}\delta_{-1}, & \\
	& \Prb\circ (\hat{X}_{-}(0))^{-1} = m_{0}, & & \Prb\circ (\hat{X}_{-}(1))^{-1} = \frac{3}{8}\delta_{1} + \frac{5}{8}\delta_{-1}, & & \Prb\circ (\hat{X}_{-}(2))^{-1} = m_{0}. &
\end{align*}
It follows that for $\Prb$-almost all $\omega \in \Omega$,
\begin{align*}
\Prb\left( X(0)\in \cdot \;|\; \mathcal{F}^{\mu} \right)(\omega) &= m_{0}, \\
\Prb\left( X(1)\in \cdot \;|\; \mathcal{F}^{\mu} \right)(\omega) &= \begin{cases}
	\frac{5\beta_{1} + 4\beta_{2}}{8(\beta_{1} + \beta_{2})} \delta_{1} + \frac{3\beta_{1} + 4\beta_{2}}{8(\beta_{1} + \beta_{2})} \delta_{-1} &\text{if } \mu_{\omega} = (m_{0},m_{1}^{+},m_{2}^{+}), \\
	\frac{5\beta_{3} + 4\beta_{4}}{8(\beta_{3} + \beta_{4})} \delta_{1} + \frac{3\beta_{3} + 4\beta_{4}}{8(\beta_{3} + \beta_{4})} \delta_{-1} &\text{if } \mu_{\omega} = (m_{0},m_{1}^{+},m_{0}), \\
	\frac{3\beta_{1} + 4\beta_{2}}{8(\beta_{1} + \beta_{2})} \delta_{1} + \frac{5\beta_{1} + 4\beta_{2}}{8(\beta_{1} + \beta_{2})} \delta_{-1} &\text{if } \mu_{\omega} = (m_{0},m_{1}^{-},m_{2}^{-}), \\
	\frac{3\beta_{3} + 4\beta_{4}}{8(\beta_{3} + \beta_{4})} \delta_{1} + \frac{5\beta_{3} + 4\beta_{4}}{8(\beta_{3} + \beta_{4})} \delta_{-1} &\text{if } \mu_{\omega} = (m_{0},m_{1}^{-},m_{0}),
\end{cases} \\
\Prb\left( X(2)\in \cdot \;|\; \mathcal{F}^{\mu} \right)(\omega) &= \begin{cases}
	\frac{21\beta_{1} + 16\beta_{2}}{32(\beta_{1} + \beta_{2})} \delta_{1} + \frac{11\beta_{1} + 16\beta_{2}}{32(\beta_{1} + \beta_{2})} \delta_{-1} &\text{if } \mu_{\omega} = (m_{0},m_{1}^{+},m_{2}^{+}), \\
	\frac{\beta_{3} + \beta_{4}}{2(\beta_{3} + \beta_{4})} \delta_{1} + \frac{\beta_{3} + \beta_{4}}{2(\beta_{3} + \beta_{4})} \delta_{-1} &\text{if } \mu_{\omega} = (m_{0},m_{1}^{+},m_{0}), \\
	\frac{11\beta_{1} + 16\beta_{2}}{32(\beta_{1} + \beta_{2})} \delta_{1} + \frac{21\beta_{1} + 16\beta_{2}}{32(\beta_{1} + \beta_{2})} \delta_{-1} &\text{if } \mu_{\omega} = (m_{0},m_{1}^{-},m_{2}^{-}), \\
	\frac{\beta_{3} + \beta_{4}}{2(\beta_{3} + \beta_{4})} \delta_{1} + \frac{\beta_{3} + \beta_{4}}{2(\beta_{3} + \beta_{4})} \delta_{-1} &\text{if } \mu_{\omega} = (m_{0},m_{1}^{-},m_{0}).
\end{cases}
\end{align*}
Thanks to the choice of $\beta_{i}$, $i\in \{1,\ldots,4\}$, according to \eqref{EqExmplBeta}, we find that the consistency condition of Definition~\ref{DefCE} is satisfied.

As to optimality, let $u \in \mathcal{U}$ be any strategy modification. Since $\Phi$ takes values in $\{\phi_{o}, \phi_{+}, \hat{\phi}_{+}, \phi_{-}, \hat{\phi}_{-}\}$ with probability one, we set
\begin{align*}
	& \psi_{o} \doteq u(\phi_{o}), & & \psi_{+} \doteq u(\phi_{+}), & & \hat{\psi}_{+} \doteq u(\hat{\phi}_{+}), & & \psi_{-} \doteq u(\phi_{-}), & & \hat{\psi}_{-} \doteq u(\hat{\phi}_{-}). &
\end{align*}  
Let $Y_{o}$, $Y_{+}$, $\hat{Y}_{+}$, $Y_{-}$, $\hat{Y}_{-}$ be the corresponding state processes, all starting from $X(0)$, hence with initial distribution $m_{0}$. 
Using that $\mathbf{M}(m_{0}) = 0$, while $\mathbf{M}(m_{t}^{+}) = -\mathbf{M}(m_{t}^{-})$, $t \in \{1, 2\}$, we obtain
\begin{align*}
	&J(\mathfrak{m}_{0};\rho,u) \\
&= \begin{aligned}[t]
	\Bigl(& c_{0}\cdot \Prb(\psi_{+}(0,Y_{+}(0)) = 1) + c_{1}\cdot \Prb(\psi_{+}(1,Y_{+}(1)) = 1) + \left( 1 - 2\Prb(Y_{+}(1) = 1) \right)\cdot \mathbf{M}(m_{1}^{+}) \\
	& + \left(1 - 2\Prb(Y_{+}(2) = 1) \right)\cdot \mathbf{M}(m_{2}^{+}) \Bigr)\cdot \beta_{1}
\end{aligned}\\
	&+ \begin{aligned}[t]
	\Bigl(& c_{0}\cdot \Prb(\hat{\psi}_{+}(0,\hat{Y}_{+}(0)) = 1) + c_{1}\cdot \Prb(\hat{\psi}_{+}(1,\hat{Y}_{+}(1)) = 1) \\
	&+ \left( 1 - 2\Prb(\hat{Y}_{+}(1) = 1) \right)\cdot \mathbf{M}(m_{1}^{+}) \Bigr)\cdot \beta_{3}
\end{aligned}\\
&+ \begin{aligned}[t]
	\Bigl(& c_{0}\cdot \Prb(\psi_{-}(0,Y_{-}(0)) = 1) + c_{1}\cdot \Prb(\psi_{-}(1,Y_{-}(1)) = 1) + \left( 1 - 2\Prb(Y_{-}(1) = 1) \right)\cdot \mathbf{M}(m_{1}^{-}) \\
	&+ \left(1 - 2\Prb(Y_{-}(2) = 1) \right)\cdot \mathbf{M}(m_{2}^{-}) \Bigr)\cdot \beta_{1}
\end{aligned}\\
&+ \begin{aligned}[t]
	\Bigl(& c_{0}\cdot \Prb(\hat{\psi}_{-}(0,\hat{Y}_{-}(0)) = 1) + c_{1}\cdot \Prb(\hat{\psi}_{-}(1,\hat{Y}_{-}(1)) = 1) \\
	&+ \left( 1 - 2\Prb(\hat{Y}_{-}(1) = 1) \right)\cdot \mathbf{M}(m_{1}^{-}) \Bigr) \cdot \beta_{3}
\end{aligned}\\
	&+ \left( c_{0}\cdot \Prb(\psi_{o}(0,Y_{o}(0)) = 1) + c_{1}\cdot \Prb(\psi_{o}(1,Y_{o}(1)) = 1)  \right)\cdot (2\beta_{2} + 2\beta_{4}).
\end{align*}
The last line above shows that if $c_{0}, c_{1} > 0$, then taking $\psi_{o} = \phi_{o}$ is optimal on the event $\{ \Phi = \phi_{o} \}$ since only in this case $\Prb(\psi_{o}(t,Y_{o}(t)) = 1) = 0$ for $t \in \{0,1 \}$. By symmetry of construction, it remains to show that $\phi_{+}$ is optimal on the event $\{ \Phi = \phi_{+} \}$ and $\hat{\phi}_{+}$ on $\{ \Phi = \hat{\phi}_{+} \}$ provided $c_{0}, c_{1} > 0$ are sufficiently small.

In verifying optimality, we will make use of the principle of dynamic programming. First, to show that $\phi_{+}$ is optimal on the event $\{ \Phi = \phi_{+} \}$, set, for $x\in \{-1,1\}$,
\begin{align*}
	V_{+}(2,x) &\doteq F(x,m_{2}^{+}) = -x\cdot \mathbf{M}(m_{2}^{+}), \\
	V_{+}(1,x) &\doteq \min_{\gamma\in \{0,1\}} \left\{ c_{1}\cdot \gamma - x\cdot \mathbf{M}(m_{1}^{+}) + \Mean\left[ V_{+}\bigl(2,\Psi(x,\gamma,\xi(2))\bigr)  \right] \right\}, \\
	V_{+}(0,x) &\doteq \min_{\gamma\in \{0,1\}} \left\{ c_{0}\cdot \gamma + \Mean\left[ V_{+}\bigl(1,\Psi(x,\gamma,\xi(1))\bigr)  \right] \right\}.
\end{align*}
As we are interested in finding optimal control actions, in defining $V_{+}$ we have omitted the weight factor $\beta_{1}$. This corresponds to computing costs with respect to the conditional probability $\Prb(\, .\,|\, \Phi = \phi_{+})$. Notice that $V_{+}$ is the value function of the optimal control problem the representative player faces when being told to play $\phi_{+}$ by the mediator. We have
\begin{align*}
	V_{+}(2,1) = -\frac{5\beta_{1}}{16(\beta_{1} + \beta_{2})},& & V_{+}(2,-1) = \frac{5\beta_{1}}{16(\beta_{1} + \beta_{2})} & 
\end{align*}
by choice of $m_{2}^{+}$, hence, also recalling $m_{1}^{+}$ and the definition of $\Psi$,
\begin{align*}
V_{+}(1,1) &= \begin{aligned}[t]
	 \min \Bigl\{&-\frac{\beta_{1}}{4(\beta_{1} + \beta_{2})} + \frac{1}{2} V_{+}(2,1) + \frac{1}{2} V_{+}(2,-1), \\
	&c_{1} -  \frac{\beta_{1}}{4(\beta_{1} + \beta_{2})} + \frac{3}{4} V_{+}(2,1) + \frac{1}{4} V_{+}(2,-1) \bigr) \Bigr\}
\end{aligned} \\
	&= \min \left\{-\frac{\beta_{1}}{4(\beta_{1} + \beta_{2})},\;
	c_{1} -  \frac{\beta_{1}}{4(\beta_{1} + \beta_{2})} - \frac{5\beta_{1}}{32(\beta_{1} + \beta_{2})} \right\}, \\
V_{+}(1,-1) &= \begin{aligned}[t]
	\min \Bigl\{&\frac{\beta_{1}}{4(\beta_{1} + \beta_{2})} + \frac{1}{2} V_{+}(2,1) + \frac{1}{2} V_{+}(2,-1), \\
	&c_{1} + \frac{\beta_{1}}{4(\beta_{1} + \beta_{2})} + \frac{1}{4} V_{+}(2,1) + \frac{3}{4} V_{+}(2,-1) \Bigr\}
\end{aligned} \\
	&= \min \left\{\frac{\beta_{1}}{4(\beta_{1} + \beta_{2})},\;
	c_{1} + \frac{\beta_{1}}{4(\beta_{1} + \beta_{2})} + \frac{5\beta_{1}}{32(\beta_{1} + \beta_{2})} \right\}.
\end{align*}
Above, the first expression inside the min corresponds to control action $\gamma = 0$, the second to $\gamma = 1$. We see that when being in state $x = -1$ at time $t = 1$, it is optimal to choose $\gamma = 0$, while when being in state $x = 1$, it is optimal to choose $\gamma = 1$ provided that
\begin{equation} \label{EqExmplC1}
	0 < c_{1} < \frac{5\beta_{1}}{32(\beta_{1} + \beta_{2})}.
\end{equation}
Assume from now on that \eqref{EqExmplC1} holds. Then $\phi_{+}$ gives the optimal control actions at time $t = 1$. As to time $t = 0$, we have
\begin{align*}
	V_{+}(0,1) &= \min \left\{ \frac{1}{2} V_{+}(1,1) + \frac{1}{2} V_{+}(1,-1),\;
c_{0} + \frac{3}{4} V_{+}(1,1) + \frac{1}{4} V_{+}(1,-1) \right\} \\
	&= \min \left\{  -\frac{1}{2}\left( \frac{5\beta_{1}}{32(\beta_{1} + \beta_{2})} - c_{1} \right),\;
	c_{0} - \frac{3}{4}\left(\frac{5\beta_{1}}{32(\beta_{1} + \beta_{2})} - c_{1} \right) - \frac{\beta_{1}}{8(\beta_{1} + \beta_{2})} \right\}, \\
	V_{+}(0,-1) &= \min \left\{\frac{1}{2} V_{+}(1,1) + \frac{1}{2} V_{+}(1,-1),\;
c_{0} + \frac{1}{4} V_{+}(1,1) + \frac{3}{4} V_{+}(1,-1) \bigr) \right\} \\
	&= \min \left\{  -\frac{1}{2}\left( \frac{5\beta_{1}}{32(\beta_{1} + \beta_{2})} - c_{1} \right),\;
	c_{0} - \frac{1}{4}\left(\frac{5\beta_{1}}{32(\beta_{1} + \beta_{2})} - c_{1} \right) + \frac{\beta_{1}}{8(\beta_{1} + \beta_{2})} \right\}.
\end{align*}
We see that when being in state $x = -1$ at time $t = 0$, it is optimal to choose $\gamma = 0$, while when being in state $x = 1$, it is certainly optimal to choose $\gamma = 1$ if
\begin{equation} \label{EqExmplC0}
	0 < c_{0} < \frac{\beta_{1}}{8(\beta_{1} + \beta_{2})}.
\end{equation}
Assume from now on that \eqref{EqExmplC0} holds. Then $\phi_{+}$ gives the optimal control actions also at time $t = 0$. Thus, under \eqref{EqExmplC1} and \eqref{EqExmplC0}, $\phi_{+}$ is optimal on $\{ \Phi = \phi_{+} \}$.

We proceed similarly to verify that $\hat{\phi}_{+}$ is optimal on the event $\{ \Phi = \hat{\phi}_{+} \}$. For $x\in \{-1,1\}$, set
\begin{align*}
\hat{V}_{+}(2,x) &\doteq F(x,m_{0}) = 0, \\
\hat{V}_{+}(1,x) &\doteq \min_{\gamma\in \{0,1\}} \left\{ c_{1}\cdot \gamma - x\cdot \mathbf{M}(m_{1}^{+}) + \Mean\left[ \hat{V}_{+}\bigl(2,\Psi(x,\gamma,\xi(2))\bigr)  \right] \right\}, \\
\hat{V}_{+}(0,x) &\doteq \min_{\gamma\in \{0,1\}} \left\{ c_{0}\cdot \gamma + \Mean\left[ \hat{V}_{+}\bigl(1,\Psi(x,\gamma,\xi(1))\bigr)  \right] \right\}.
\end{align*}
We have again omitted the weight factor, here $\beta_{3}$. Costs are thus computed with respect to the conditional probability $\Prb(\, .\,|\, \Phi = \hat{\phi}_{+})$. Notice that $\hat{V}_{+}$ is the value function of the optimal control problem the representative player faces when being told to play $\hat{\phi}_{+}$ by the mediator. We have $\hat{V}_{+}(2,.) \equiv 0$,
\begin{align*}
	\hat{V}_{+}(1,1) &= \min \left\{-\frac{\beta_{1}}{4(\beta_{1} + \beta_{2})},\;
c_{1} -  \frac{\beta_{1}}{4(\beta_{1} + \beta_{2})} \right\}, \\
	\hat{V}_{+}(1,-1) &= \min \left\{\frac{\beta_{1}}{4(\beta_{1} + \beta_{2})},\;
c_{1} + \frac{\beta_{1}}{4(\beta_{1} + \beta_{2})} \right\}.
\end{align*}
The first expression inside the min above corresponds again to control action $\gamma = 0$, the second to $\gamma = 1$. We see that at time $t = 1$ it is always optimal to choose $\gamma = 0$. This is exactly what $\hat{\phi}_{+}$ prescribes at $t = 1$. As to time $t = 0$, we have
\begin{align*}
	\hat{V}_{+}(0,1) &= \min \left\{ \frac{1}{2} \hat{V}_{+}(1,1) + \frac{1}{2} \hat{V}_{+}(1,-1),\;
	c_{0} + \frac{3}{4} \hat{V}_{+}(1,1) + \frac{1}{4} \hat{V}_{+}(1,-1) \right\} \\
	&= \min \left\{ 0,\;
	c_{0} - \frac{\beta_{1}}{8(\beta_{1} + \beta_{2})} \right\}, \\
	\hat{V}_{+}(0,-1) &= \min \left\{\frac{1}{2} \hat{V}_{+}(1,1) + \frac{1}{2} \hat{V}_{+}(1,-1),\;
	c_{0} + \frac{1}{4} \hat{V}_{+}(1,1) + \frac{3}{4} \hat{V}_{+}(1,-1) \bigr) \right\} \\
	&= \min \left\{ 0,\;
	c_{0} + \frac{\beta_{1}}{8(\beta_{1} + \beta_{2})} \right\}.
\end{align*}
We see that when being in state $x = -1$ at time $t = 0$, it is optimal to choose $\gamma = 0$, while when being in state $x = 1$, it is optimal to choose $\gamma = 1$ if \eqref{EqExmplC0} holds. This is what $\hat{\phi}_{+}$ prescribes at time $t = 0$. Therefore, under \eqref{EqExmplC0}, $\hat{\phi}_{+}$ is optimal on $\{ \Phi = \hat{\phi}_{+} \}$.

We have thus established that the correlated flow $\rho$ defined above is a correlated solution of the mean field game with initial distribution $m_{0}$ provided the coefficients satisfy \eqref{EqExmplBeta}, \eqref{EqExmplC1}, and \eqref{EqExmplC0}. In particular, given any $\alpha \in (0,1)$, we have that $\rho$ is a solution if
\begin{align*}
	& \beta_{1} = \beta_{2} = \frac{\alpha}{4}, & & \beta_{3} = \beta_{4} = \frac{1-\alpha}{4}, & & 0 < c_{0} < \frac{1}{16}, & & 0 < c_{1} < \frac{5}{64}. &
\end{align*}

\begin{rem} \label{RemCorrelatedExmpl}
	The correlated flow constructed in the example above does not simply arise by randomizing among solutions with deterministic flow of measures. While the representative player can infer with probability one what the flow of measures will be when receiving the recommendation to play strategy $\phi_{+}$, $\hat{\phi}_{+}$, $\phi_{-}$, or $\hat{\phi}_{-}$, she cannot do so when being told to play strategy $\phi_{o}$. In this case, the final evolution of the flow of measures will be uncertain for the player not only at time zero, but also at time $t = 1$. Also notice that in our example there are multiple, actually infinitely many solutions.
\end{rem}


\section{Convergence of correlated equilibria} \label{SectConv}

For $N\in \mathbb{N}$, let $\mathfrak{m}^N \in \prbms{\mathcal{X}^{N}}$, let $\gamma^{N} \in \prbms{\mathcal{R}^{N}}$ be a strategy profile, and let $(\epsilon_{N})_{N\in \mathbb{N}} \subset [0,\infty)$. Moreover, let $\mathfrak{m}_{0}\in \prbms{\mathcal{X}}$. We make the following assumptions:
\begin{hypenv}

	\item \label{hypContSystem} Continuity property of the system function $\Psi$: There exists a measurable function $\boldsymbol{w}\!: [0,\infty) \rightarrow [0,1]$ with $\boldsymbol{w}(s) \to 0$ as $s \to 0+$ such that for every $(t,x,a) \in \{0,\ldots,T-1\}\times \mathcal{X} \times \Gamma$, all $m,\tilde{m} \in \prbms{\mathcal{X}}$,
	\[
		\int_{\mathcal{Z}} \mathbf{1}_{\Psi(t,x,m,a,z)\neq \Psi(t,x,\tilde{m},a,z)} \nu(dz) \leq \boldsymbol{w}\bigl( \mathsf{dist}(m,\tilde{m}) \bigr).
	\]
	Moreover, for every $t\in \{0,\ldots,T-1\}$, every $\tau \in \prbms{\mathcal{X} \times \prbms{\mathcal{X}}\times \Gamma}$, $\Psi(t,.)$ is $\tau\otimes \nu$-almost everywhere continuous.

	\item \label{hypContCosts} The cost coefficients $f$, $F$ are continuous.
	
	\item \label{hypCE} For every $N\in \mathbb{N}$, $\gamma^{N}$ is a symmetric $\epsilon_{N}$-correlated equilibrium in restricted strategies with initial distribution $\mathfrak{m}^N$.
	
	\item \label{hypEps} The sequence $(\epsilon_{N})_{N\in \mathbb{N}}$ converges to zero as $N\to \infty$.
	
	\item \label{hypInit} Initial distributions: $\mathfrak{m}^N = \otimes^{N} \mathfrak{m}_{0,N}$ where $\mathfrak{m}_{0,N} \to \mathfrak{m}_{0}$ as $N\to \infty$.
\end{hypenv}

\begin{rem} \label{RemContSystem}
The continuity property \hypref{hypContSystem} is satisfied, for instance, if $\Psi$ is defined as follows. Choose $L > 0$, let $d\doteq |\mathcal{X}|$ be the number of states, and let  $\sigma\!: \{1,\ldots,d \} \rightarrow \mathcal{X}$ be a bijection. For $(t,x,a) \in \{0,\ldots,T-1\}\times \mathcal{X} \times \Gamma$, choose functions $a_{1,t,x,a},\ldots,a_{d,t,x,a} \!: \prbms{\mathcal{X}} \rightarrow [0,1]$ that are $L$-Lipschitz continuous and such that $\sum_{i=1}^{d} a_{i,t,x,a} = 1$. Now set
\[
\begin{split}
	\Psi(t,x,m,a,z) &\doteq \sigma\left( \argmin\left\{ j\in \{1,\ldots,d\} : \sum_{i=1}^{j} a_{i,t,x,a}(m) \geq z \right\} \right), \\
	& (t,x,m,a,z)\in \{0,\ldots,T-1\}\times \mathcal{X} \times \prbms{\mathcal{X}} \times \Gamma\times \mathcal{Z}.
\end{split}
\]
Recall that $\nu$ indicates the uniform distribution on $\mathcal{Z} = [0,1]$. With the above definition of $\Psi$, we have for all $m,\tilde{m} \in \prbms{\mathcal{X}}$,
\begin{align*}
	\int_{\mathcal{Z}} \mathbf{1}_{\Psi(t,x,m,a,z)\neq \Psi(t,x,\tilde{m},a,z)} \nu(dz) &\leq \sum_{j=1}^{d-1} \left| \sum_{i=1}^{j} a_{i,t,x,a}(m) - \sum_{i=1}^{j} a_{i,t,x,a}(\tilde{m}) \right| \\	
	&\leq L\frac{d(d-1)}{2} \mathsf{dist}(m,\tilde{m}).
\end{align*}
The first part of \hypref{hypContSystem} is thus satisfied with $\boldsymbol{w}(s) = L\frac{d(d-1)}{2} s$. This modulus of continuity changes if the functions $a_{1,t,x,a},\ldots,a_{d,t,x,a}$ are (uniformly) continuous, but not Lipschitz. In order to check the second part of \hypref{hypContSystem}, fix $t\in \{0,\ldots,T-1\}$ and let $D_{t}$ denote the set of points of discontinuity of $\Psi(t,.)$. By construction, the functions $a_{1,t,x,a},\ldots,a_{d,t,x,a}$ are continuous on $\prbms{\mathcal{X}}$, and they depend continuously also on $x,a$ since $\mathcal{X}$, $\Gamma$ are finite sets. In view of the definition of $\Psi$, it follows that
\[
	D_{t} \subseteq	\bigcup_{j=1}^{d} \left\{ (x,m,a,z)\in \mathcal{X}\times \prbms{\mathcal{X}}\times \Gamma\times \mathcal{Z} : \sum_{i=1}^{j} a_{i,t,x,a}(m) = z \right\}.
\]
The assertion of the second part of \hypref{hypContSystem} is now a consequence of Fubini's theorem as $\nu$ assigns measure zero to any finite subset of $\mathcal{Z}$.
\end{rem}

\begin{rem}
The example from Section~\ref{SectMFG} satisfies Assumptions \hypref{hypContSystem} and \hypref{hypContCosts} above. The system function $\Psi$ there can be written as in Remark~\ref{RemContSystem}. In particular, $\Psi(t,.)$ is almost surely continuous with respect to $\tau\otimes \nu$ given any $\tau \in \prbms{\mathcal{X} \times \prbms{\mathcal{X}}\times \Gamma}$, but discontinuous as a function $\mathcal{X} \times \prbms{\mathcal{X}}\times \Gamma \times \mathcal{Z} \rightarrow \mathcal{X}$.
\end{rem}

For $N\in \mathbb{N}\setminus \{1 \}$, let
\[
	\left((\Omega_{N},\mathcal{F}_{N},\Prb_{N}),\Phi^{N}_1,\ldots,\Phi^{N}_N,X^{N}_{1}(.),\ldots,X^{N}_{N}(.),\xi^{N}_{1}(.),\ldots,\xi^{N}_{N}(.) \right)
\]
be a realization of the triple $(\mathfrak{m}^{N},\gamma^{N},\Id)$, and set
\[
	\rho^{N} \doteq \Prb_{N}\circ \left( \Phi^{N}_{1}, \mu^{N}_{1}(0),\ldots,\mu^{N}_{1}(T) \right)^{-1},
\]
where $\mu^{N}_{1}(t) = \frac{1}{N-1} \sum_{j=2}^{N} \delta_{X^{N}_{j}(t)}$, $t\in \{0,\ldots,T\}$, as above. We then have the following convergence result:

\begin{thrm} \label{ThConvergence}
	Grant \hyprefall. Then $(\rho^{N})_{N\in \mathbb{N}}$ is relatively compact as a subset of $\prbms{\mathcal{R} \times \prbms{\mathcal{X}}^{T+1}}$, and any limit point is a correlated solution of the mean field game in restricted strategies with initial distribution $\mathfrak{m}_{0}$.
\end{thrm}

\begin{proof}
If $\mathcal{S}$ is a compact Polish space, then $\prbms{\mathcal{S}}$ is compact with respect to the topology of weak convergence of measures. Since $\mathcal{X}$, $\mathcal{R}$ are finite sets, hence compact Polish spaces under the discrete topology, we have that $\prbms{\mathcal{R} \times \prbms{\mathcal{X}}^{T+1}}$ is compact. This in turn implies the relative compactness of $(\rho^{N})_{N\in \mathbb{N}}$ in $\prbms{\mathcal{R} \times \prbms{\mathcal{X}}^{T+1}}$.

In order to identify the limit points of $(\rho^{N})_{N\in \mathbb{N}}$, set
\[
	\eta^{N} \doteq \Prb_{N}\circ \left( \Phi^{N}_{1}, X^{N}_{1}(0),\ldots,X^{N}_{1}(T), \xi^{N}_{1}(1),\ldots,\xi^{N}_{1}(T), \mu^{N}_{1}(0),\ldots,\mu^{N}_{1}(T) \right)^{-1}.
\]
Clearly, $\rho^{N}$ coincides with the image (push forward) measure of $\eta^{N}$ under the natural projection $\mathcal{R} \times \mathcal{X}^{T+1} \times \mathcal{Z}^{T} \times \prbms{\mathcal{X}}^{T+1} \rightarrow \mathcal{R} \times \prbms{\mathcal{X}}^{T+1}$. Moreover, $(\eta^{N})_{N\in \mathbb{N}}$ is relatively compact in $\prbms{\mathcal{R} \times \mathcal{X}^{T+1} \times \mathcal{Z}^{T} \times \prbms{\mathcal{X}}^{T+1}}$ since the space $\mathcal{R} \times \mathcal{X}^{T+1} \times \mathcal{Z}^{T} \times \prbms{\mathcal{X}}^{T+1}$ is compact, too. To identify the limit points of $(\rho^{N})_{N\in \mathbb{N}}$ it is therefore enough to characterize the limits of convergent subsequences of $(\eta^{N})_{N\in \mathbb{N}}$. We will proceed in several steps.

\paragraph{Step One.} Let $(\eta^{N_k})_{k\in \mathbb{N}}$ be any convergent subsequence of $(\eta^{N})_{N\in \mathbb{N}}$, and denote its limit by $\eta$. Let $(\Phi,X(.),\xi(.),\mu(.))$ be a $\mathcal{R} \times \mathcal{X}^{T+1} \times \mathcal{Z}^{T} \times \prbms{\mathcal{X}}^{T+1}$-valued random element on some probability space $(\Omega,\mathcal{F},\Prb)$ such that
\[
	\eta = \Prb\circ \left( \Phi, X(0),\ldots,X(T), \xi(1),\ldots,\xi(T), \mu(0),\ldots,\mu(T) \right)^{-1}.
\]
Set
\[
	\rho \doteq \Prb\circ \left( \Phi, \mu(0),\ldots,\mu(T) \right)^{-1}.
\]
Then the following properties hold:
\begin{enumerate}[(a)]
	\item\label{ProofConvInitial} $\Prb\circ (X(0))^{-1} = \mathfrak{m}_{0}$;
	
	\item\label{ProofConvNoise} $\xi(t)$, $t\in \{1,\ldots,T\}$, are i.i.d.\ with common distribution $\nu$;
	
	\item\label{ProofConvInd} $\xi(.)$, $X(0)$, and $(\Phi,\mu(.))$ are independent;
	
	\item\label{ProofConvLimitDyn} $\Prb$-almost surely, for every $t\in \{0,\ldots,T-1\}$,
	\[
		X(t+1) = \Psi\left( t, X(t), \mu(t), \Phi(t,X(t)), \xi(t+1) \right);
	\]
	
	\item\label{ProofConvCosts} $\lim_{k\to \infty} J^{N_k}_{1}\left( \mathfrak{m}^{N_k};\gamma^{N_k},\Id \right) = J(\mathfrak{m}_{0};\rho,\Id)$.
\end{enumerate}

Point~\eqref{ProofConvInitial} is a consequence of assumption \hypref{hypInit}. 
Point~\eqref{ProofConvNoise} follows from the corresponding independence property of $\xi^{N}_{1}(t)$, $t\in \{1,\ldots,T\}$, and their joint convergence in distribution.

Points \eqref{ProofConvInd} and \eqref{ProofConvLimitDyn} will follow from analogous properties established in the next two steps by taking $u=\Id$. The convergence of costs according to \eqref{ProofConvCosts} is again a consequence of convergence in distribution in conjunction with assumption~\hypref{hypContCosts}.

\paragraph{Step Two.} Let $u\!: \mathcal{R} \rightarrow \mathcal{R}$ be any strategy modification. We define a realization of $(\mathfrak m_0, \rho, u)$ with the same noises and in the same probability space as the realization of $(\mathfrak m_0, \rho, \Id)$ given in Step One. For $N\in \mathbb{N}\setminus \{ 1\}$, define $\mathcal{X}$-valued random variables $\tilde{X}^{N}_{j}(t)$, $j\in \{1,\ldots,N\}$, $t\in \{0,\ldots,T\}$, on $(\Omega_{N},\mathcal{F}_{N},\Prb_{N})$ recursively through
\begin{align*}
	\tilde{X}^{N}_{i}(0) &\doteq X^{N}_{i}(0) \text{ for every } i \in \{1,\ldots,N\}, \\
	\tilde{X}^{N}_{1}(t+1) &\doteq \Psi\left(t,\tilde{X}^{N}_{1}(t),\tilde{\mu}^{N}_{1}(t),u\circ\Phi^{N}_{1}\left(t,\tilde{X}^{N}_{1}(t)\right), \xi^{N}_{1}(t+1) \right),\\
		\tilde{X}^{N}_{j}(t+1) &= \Psi\left(t,\tilde{X}^{N}_{j}(t),\tilde{\mu}^{N}_{j}(t),\Phi^{N}_{j}\left(t,\tilde{X}^{N}_{j}(t)\right), \xi^{N}_{j}(t+1) \right),\quad j\neq 1, \\
		&t\in \{0,\ldots,T-1\},
\end{align*}
where
\[
	\tilde{\mu}^{N}_{i}(t) \doteq \frac{1}{N-1} \sum_{l\neq i} \delta_{\tilde{X}^{N}_{l}(t)},\quad i\in \{1,\ldots,N\}.
\]
Set
\[
	\tilde{\eta}^{N} \doteq \Prb_{N}\circ \left( \Phi^{N}_{1}, \tilde{X}^{N}_{1}(0),\ldots,\tilde{X}^{N}_{1}(T), \xi^{N}_{1}(1),\ldots,\xi^{N}_{1}(T), \tilde{\mu}^{N}_{1}(0),\ldots,\tilde{\mu}^{N}_{1}(T) \right)^{-1}.
\]
Reasoning as in Step~One, we have that $(\tilde{\eta}^{N})_{N\in \mathbb{N}}$ is relatively compact in $\prbms{\mathcal{R} \times \mathcal{X}^{T+1} \times \mathcal{Z}^{T} \times \prbms{\mathcal{X}}^{T+1}}$, and so is $(\tilde{\eta}^{N_k})_{k\in \mathbb{N}}$. Choose any convergent subsequence of $(\tilde{\eta}^{N_k})_{k\in \mathbb{N}}$, which we continue to indicate by $(\tilde{\eta}^{N_k})_{k\in \mathbb{N}}$, thus omitting the sub-subscript. Denote its limit by $\tilde{\eta}$, and let $(\tilde{\Phi},\tilde{X}(.),\tilde{\xi}(.),\tilde{\mu}(.))$ be a $\mathcal{R} \times \mathcal{X}^{T+1} \times \mathcal{Z}^{T} \times \prbms{\mathcal{X}}^{T+1}$-valued random element on some probability space $(\tilde{\Omega},\tilde{\mathcal{F}},\tilde{\Prb})$ such that
\[
	\tilde{\eta} = \tilde{\Prb}\circ \left( \tilde{\Phi}, \tilde{X}(0),\ldots, \tilde{X}(T), \tilde{\xi}(1),\ldots,\tilde{\xi}(T), \tilde{\mu}(0),\ldots, \tilde{\mu}(T) \right)^{-1}.
\]
Set
\[
	\tilde{\rho} \doteq \tilde{\Prb}\circ \left( \tilde{\Phi}, \tilde{\mu}(0),\ldots,\tilde{\mu}(T) \right)^{-1}.
\]
Then the following properties hold:
\begin{enumerate}[(a)]
	\item $\Prb\circ (\tilde{X}(0))^{-1} = \mathfrak{m}_{0}$;
		
	\item $\tilde{\xi}(t)$, $t\in \{1,\ldots,T\}$, are i.i.d.\ with common distribution $\nu$;
	
	\item $\tilde{\xi}(.)$, $\tilde{X}(0)$, and $(\tilde{\Phi},\tilde{\mu}(.))$ are independent;

	\item $\tilde{\Prb}$-almost surely, for every $t\in \{0,\ldots,T-1\}$,
	\[
		\tilde{X}(t+1) = \Psi\left( t, \tilde{X}(t), \tilde{\mu}(t), u\circ \tilde{\Phi}(t,\tilde{X}(t)), \tilde{\xi}(t+1) \right);
	\]
	
	\item $\lim_{k\to \infty} J^{N_k}_{1}\left( \mathfrak{m}^{N_k};\gamma^{N_k},u \right) = J(\mathfrak{m}_{0};\tilde{\rho},u)$.
\end{enumerate}

Points \eqref{ProofConvInitial}, \eqref{ProofConvNoise}, and \eqref{ProofConvCosts} follow as in Step~One. The independence property~\eqref{ProofConvInd} will be established in Step~Three. To verify Property~\eqref{ProofConvLimitDyn}, define functions $G_{t}\!: \mathcal{R}\times \mathcal{X}\times \prbms{\mathcal{X}} \times \mathcal{Z} \rightarrow \mathcal{X}$, $t\in \{1,\ldots,T\}$, by setting
\[
	G_{t}(\phi,x,m,z) \doteq \Psi\bigl(t,x,m, u(\phi)(t,x), z\bigr).
\]
The function $G_{t}$ is $\sigma\otimes \nu$-almost everywhere continuous given any measure $\sigma\in \prbms{\mathcal{R}\times \mathcal{X} \times \prbms{\mathcal{X}}}$. This follows from the second part of assumption~\hypref{hypContSystem} and the fact that the spaces $\mathcal{R}$ and $\mathcal{X}$ are finite. In particular, the mapping $\mathcal{R}\times \mathcal{X} \rightarrow \Gamma$ given by $(\phi,x) \mapsto \phi(t,x)$ is continuous for every $t\in \{0,\ldots,T-1\}$. Property~\eqref{ProofConvLimitDyn} is now a consequence of Lemma~\ref{LemmaEquationConv}, the almost everywhere continuity of $G_{t}$ in conjunction with the independence properties \eqref{ProofConvNoise} and \eqref{ProofConvInd}, and the convergence in distribution of
\[
	\left(\tilde{X}^{N_{k}}_{1}(t+1),\left(\Phi^{N_{k}}_{1}, \tilde{X}^{N_{k}}_{1}(t), \tilde{\mu}^{N_{k}}_{1}(t), \xi^{N_{k}}_{1}(t+1)\right)\right)
\]
to
\[
	\left( \tilde{X}(t+1), \left(\tilde{\Phi}, \tilde{X}(t), \tilde{\mu}(t), \tilde{\xi}(t+1) \right) \right)
\]
as $k\to \infty$, by the mapping theorem.

Point~\eqref{ProofConvCosts} above, together with the corresponding property established in Step~One, entails thanks to assumptions \hypref{hypCE} and \hypref{hypEps} that
\[
	J(\mathfrak{m}_{0};\rho,\Id) \leq J(\mathfrak{m}_{0};\tilde{\rho},u).
\]
It remains to show that $\rho = \tilde{\rho}$ and that property~\eqref{ProofConvInd} holds.

\paragraph{Step Three.} For $N\in \mathbb{N}\setminus \{ 1\}$, recursively define $\mathcal{X}$-valued random variables $\hat{X}^{N}_{j}(t)$, $j\in \{2,\ldots,N\}$, $t\in \{0,\ldots,T\}$, on $(\Omega_{N},\mathcal{F}_{N},\Prb_{N})$ through
\begin{align*}
\hat{X}^{N}_{j}(0) &\doteq X^{N}_{j}(0) \text{ for every } j \in \{2,\ldots,N\}, \\
\hat{X}^{N}_{j}(t+1) &= \Psi\left(t,\hat{X}^{N}_{j}(t),\hat{\mu}^{N}_{1}(t),\Phi^{N}_{j}\left(t,\hat{X}^{N}_{j}(t)\right), \xi^{N}_{j}(t+1) \right),\quad j\neq 1, \\
&t\in \{0,\ldots,T-1\},
\end{align*}
where
\begin{equation}\label{eq:emp-meas-2-N}
\hat{\mu}^{N}_{1}(t) \doteq \frac{1}{N-1} \sum_{l=2}^{N} \delta_{\hat{X}^{N}_{l}(t)}.
\end{equation}
The main difference between $\tilde X$ and $\hat X$ is that for the latter we consider only players in $\{2,\ldots,N\}$, thus excluding the contribution of Player~1; see the empirical measure \eqref{eq:emp-meas-2-N} above.

Let $\mathsf{dist}$ be the metric on $\prbms{\mathcal{X}}$ introduced in Section~\ref{SectNotation}. We claim that for every $t\in \{0,\ldots,T\}$,
\begin{equation} \label{ProofConvHatTilde}
	\frac{1}{N-1} \sum_{j=2}^{N}\Prb_{N}\left( \hat{X}^{N}_{j}(t) \neq \tilde{X}^{N}_{j}(t) \right)  \stackrel{N\to \infty}{\longrightarrow} 0.
\end{equation}
We verify \eqref{ProofConvHatTilde} by induction over $t\in \{0,\ldots,T\}$. First notice that, by inequality~\eqref{EqEMdist}, for all $j\in \{ 1,\ldots,N \}$, including $j=1$, all $t\in \{0,\ldots,T\}$,
\begin{equation} \label{ProofConvHatTilde2}
	\Mean_{N}\left[ \mathsf{dist}\bigl(\hat{\mu}^{N}_{1}(t), \tilde{\mu}^{N}_{j}(t)\bigr) \right] \leq \frac{1}{N-1} + \frac{1}{N-1} \sum_{l=2}^{N}\Prb_{N}\left( \hat{X}^{N}_{l}(t) \neq \tilde{X}^{N}_{l}(t) \right).
\end{equation}
It is clear that relation~\eqref{ProofConvHatTilde} holds if $t=0$, since $\hat{X}^{N}_{j}(0) = \tilde{X}^{N}_{j}(0) = X^{N}_{j}(0)$ for all $j\in \{2,\ldots,N\}$. Now, suppose that \eqref{ProofConvHatTilde} holds for some $t\in \{0,\ldots,T-1\}$. For each $j\in \{2,\ldots,N\}$, we have
\begin{multline*}
	\Prb_{N}\left( \hat{X}^{N}_{j}(t+1) \neq \tilde{X}^{N}_{j}(t+1) \right) \leq \Prb_{N}\left( \hat{X}^{N}_{j}(t) \neq \tilde{X}^{N}_{j}(t) \right) \\
	+ \Mean_{N}\left[ \mathbf{1}_{\Psi\left(t,\hat{X}^{N}_{j}(t),\hat{\mu}^{N}_{1}(t),\Phi^{N}_{j}\left(t,\hat{X}^{N}_{j}(t)\right), \xi^{N}_{j}(t+1) \right) \neq \Psi\left(t,\hat{X}^{N}_{j}(t),\tilde{\mu}^{N}_{j}(t),\Phi^{N}_{j}\left(t,\hat{X}^{N}_{j}(t)\right), \xi^{N}_{j}(t+1) \right) } \right].
\end{multline*}
Using the Fubini-Tonelli theorem, the independence of $\xi^{N}_{j}(t+1)$, as well as assumption \hypref{hypContSystem}, we find the expected value in the display above to be less than or equal to
\begin{multline*}
	\Mean_{N}\left[ \int_{\mathcal{Z}} \mathbf{1}_{\Psi\left(t,\hat{X}^{N}_{j}(t),\hat{\mu}^{N}_{1}(t),\Phi^{N}_{j}\left(t,\hat{X}^{N}_{j}(t) \right), z \right) \neq \Psi\left(t,\hat{X}^{N}_{j}(t),\tilde{\mu}^{N}_{j}(t),\Phi^{N}_{j}\left(t,\hat{X}^{N}_{j}(t) \right), z \right) } \nu(dz) \right] \\
	\leq \Mean_{N}\left[ \boldsymbol{w}\left( \mathsf{dist}\bigl(\hat{\mu}^{N}_{1}(t), \tilde{\mu}^{N}_{j}(t)\bigr) \right) \right].
\end{multline*}
The induction hypothesis and \eqref{ProofConvHatTilde2} imply that
\[
	\max_{j\in \{ 2,\ldots,N \}} \Mean_{N}\left[ \mathsf{dist}\bigl(\hat{\mu}^{N}_{1}(t), \tilde{\mu}^{N}_{j}(t)\bigr) \right] \stackrel{N\to \infty}{\longrightarrow} 0.
\]
This in turn entails, by Markov's inequality and the fact that $\boldsymbol{w}$ is bounded non-negative with $\boldsymbol{w}(s) \to 0$ as $s \to 0+$, that
\[
	\max_{j\in \{ 2,\ldots,N \}} \Mean_{N}\left[ \boldsymbol{w}\left( \mathsf{dist}\bigl(\hat{\mu}^{N}_{1}(t), \tilde{\mu}^{N}_{j}(t)\bigr) \right) \right] \stackrel{N\to\infty}{\longrightarrow} 0.
\]
Using again the induction hypothesis, we find that
\begin{multline*}
	\frac{1}{N-1} \sum_{j=2}^{N}\Prb_{N}\left( \hat{X}^{N}_{j}(t+1) \neq \tilde{X}^{N}_{j}(t+1) \right) \\
	\leq \frac{1}{N-1} \sum_{j=2}^{N}\Prb_{N}\left( \hat{X}^{N}_{j}(t) \neq \tilde{X}^{N}_{j}(t) \right) + \max_{j\in \{ 2,\ldots,N \}} \Mean_{N}\left[ \boldsymbol{w}\left( \mathsf{dist}\bigl(\hat{\mu}^{N}_{1}(t), \tilde{\mu}^{N}_{j}(t)\bigr) \right) \right] \stackrel{N\to \infty}{\longrightarrow} 0.
\end{multline*}
This establishes \eqref{ProofConvHatTilde} for all $t\in \{0,\ldots,T \}$.

As a consequence of \eqref{ProofConvHatTilde} and \eqref{ProofConvHatTilde2}, we obtain
\[
	\Mean_{N}\left[\sum_{t=0}^{T} \mathsf{dist}\bigl(\hat{\mu}^{N}_{1}(t), \tilde{\mu}^{N}_{1}(t)\bigr) \right] \stackrel{N\to \infty}{\longrightarrow} 0.
\]
By Lemma~\ref{LemmaPerturbConv}, this implies that the sequences  
\[ (\Phi^{N_{k}}_{1},\tilde{X}^{N_{k}}_{1}(.),\xi^{N_{k}}_{1}(.),\tilde{\mu}^{N_{k}}_{1}(.)),\quad\text{and}\quad (\Phi^{N_{k}}_{1},\tilde{X}^{N_{k}}_{1}(.),\xi^{N_{k}}_{1}(.),\hat{\mu}^{N_{k}}_{1}(.))\]
have the same limit in distribution, namely $(\tilde{\Phi},\tilde{X}(.),\tilde{\xi}(.),\tilde{\mu}(.))$. From equation \eqref{eq:emp-meas-2-N} it is clear that the definition of $\hat{\mu}^{N}_{1}(.)$ through the random variables $\hat{X}^{N}_{j}(t)$, with $j\geq 2$, does not depend on the strategy modification $u$. Thus, also $(\Phi^{N_{k}}_{1},X^{N_{k}}_{1}(.),\xi^{N_{k}}_{1}(.),\mu^{N_{k}}_{1}(.))$ and $(\Phi^{N_{k}}_{1},X^{N_{k}}_{1}(.),\xi^{N_{k}}_{1}(.),\hat{\mu}^{N_{k}}_{1}(.))$ have the same limit in distribution, namely $(\Phi,X(.),\xi(.),\mu(.))$. Since $X^{N}_{1}(0) = \tilde{X}^{N}_{1}(0)$ for every $N \in \mathbb{N}$, we find that
\[
	\Prb\circ \left(\Phi,X(0),\xi(.),\mu(.) \right)^{-1} = \tilde{\Prb}\circ \left(\tilde{\Phi},\tilde{X}(0),\tilde{\xi}(.),\tilde{\mu}(.) \right)^{-1}.
\]
This implies, in particular, that $\rho = \tilde{\rho}$. In addition, by the fact that $\hat{\mu}^{N}_{1}(.)$ does not depend on Player~1's state (see \eqref{eq:emp-meas-2-N}), the independence of $X^{N}_{1}(0),\ldots,X^{N}_{N}(0)$ according to \hypref{hypInit}, and the independence of $X^{N}_{j}(0)$, $\xi^{N}_{j}(.)$, $\Phi^{N}_{j}$, $j\in \{ 1,\ldots,N \}$, we have that
\[
	X^{N}_{1}(0),\; \xi^{N}_{1}(.), \text{ and } (\Phi^{N}_{1},\hat{\mu}^{N}_{1}) \text{ are independent, for every } N \in \mathbb{N}.
\]
Convergence in distribution now yields property~\eqref{ProofConvInd}. Thanks to Step~Two, it follows that
\[
	J(\mathfrak{m}_{0};\rho,\Id) \leq J(\mathfrak{m}_{0};\rho,u).
\]
The optimality condition of Definition~\ref{DefCE} is therefore satisfied.

\paragraph{Step Four.} We verify the consistency condition in Definition~\ref{DefCE}. For $N\in \mathbb{N}$, let $\mu^{N}(t)$ denote the empirical measure of the states at time $t\in \{0,\ldots,T\}$ of all players in the $N$-player game, and let $\boldsymbol{\mu}^{N}$ denote the empirical measure of their state trajectories:
\begin{align*}
	& \mu^{N}(t) \doteq \frac{1}{N} \sum_{i=1}^{N} \delta_{X^{N}_{i}(t)}, & & \boldsymbol{\mu}^{N} \doteq \frac{1}{N} \sum_{i=1}^{N} \delta_{(X^{N}_{i}(0),\ldots,X^{N}_{i}(T))}. &
\end{align*}
Thus, $\mu^{N}(t)$ is a $\prbms{\mathcal{X}}$-valued random variable, for every $t$, while $\boldsymbol{\mu}^{N}$ is a $\prbms{\mathcal{X}^{T+1}}$-valued random variable.

As a consequence of the symmetry of the correlated profiles according to \hypref{hypCE}, of the initial distributions according to \hypref{hypInit}, and of the dynamics, we obtain that $(X^{N}_{1}(t),\ldots,X^{N}_{N}(t))$ is a finite exchangeable sequence of $\mathcal{X}$-valued random variables for every $t$, while $(X^{N}_{1}(.),\ldots,X^{N}_{N}(.))$ is a finite exchangeable sequence of $\mathcal{X}^{T+1}$-valued random variables. Lemma~\ref{LemmaCondDistExchange} now yields the conditional distributions of the state and of the state trajectory of player~1 given the corresponding empirical measure:
\begin{align*}
	& \Prb_{N} \left( X^{N}_{1}(t)\in . \;|\; \mu^{N}(t) \right) = \mu^{N}(t)(.), & &t\in \{0,\ldots,T\}, & \\
	& \Prb_{N} \left( (X^{N}_{1}(0),\ldots,X^{N}_{1}(T)) \in . \;|\; \boldsymbol{\mu}^{N} \right) = \boldsymbol{\mu}^{N}(.). & &&
\end{align*}
Applying the conditional distribution of the state trajectory of player~1 to sets of the form $\mathcal{X}^{t} \times B \times \mathcal{X}^{T-t}$ shows that also
\[
	\Prb_{N} \left( X^{N}_{1}(t)\in . \;|\; \boldsymbol{\mu}^{N} \right) = \mu^{N}(t)(.) \text{ for every } t\in \{0,\ldots,T\}.
\]
The $\sigma$-algebra generated by $\mu^{N}(.) = (\mu^{N}(0),\ldots,\mu^{N}(T))$, the flow of empirical measures, is contained in $\sigma(\boldsymbol{\mu}^{N})$, while it contains $\sigma(\mu^{N}(t))$ for every $t$. In view of Lemma~\ref{LemmaCondDistTower}, we thus find that
\[
	\Prb_{N} \left( X^{N}_{1}(t)\in . \;|\; (\mu^{N}(0),\ldots,\mu^{N}(T)) \right) = \mu^{N}(t)(.) \text{ for every } t\in \{0,\ldots,T\}.
\]
According to Step~One (using the mapping theorem), we have that, for every $t\in \{0,\ldots,T\}$, the random vector $(X^{N_{k}}_{1}(t),(\mu^{N_{k}}_{1}(0),\ldots,\mu^{N_{k}}_{1}(T)),\mu^{N_{k}}_{1}(t))$ converges in distribution to $(X(t),(\mu(0),\ldots,\mu(T)),\mu(t))$ as $k\to \infty$. Now, for every $k\in \mathbb{N}$, every $\omega\in \Omega_{N_k}$, every $t\in \{0,\ldots,T\}$,
\[
	\mathsf{dist}\left( \mu^{N_{k}}_{1,\omega}(t), \mu^{N_{k}}_{\omega}(t) \right) \leq (N_{k}-1)\left( \frac{1}{N_{k}-1} - \frac{1}{N_{k}} \right) + \frac{1}{N_{k}} = \frac{2}{N_{k}},
\]
where $\mathsf{dist}$ is the metric on $\prbms{\mathcal{X}}$ introduced in Section~\ref{SectNotation}. This implies by Lemma~\ref{LemmaPerturbConv} that also the vector $(X^{N_{k}}_{1}(t),(\mu^{N_{k}}(0),\ldots,\mu^{N_{k}}(T)),\mu^{N_{k}}(t))$ converges in distribution to $(X(t),(\mu(0),\ldots,\mu(T)),\mu(t))$ as $k\to \infty$. By Lemma~\ref{LemmaCondDistConv}, we now find the conditional distribution of $X(t)$ given the flow of measures $(\mu(0),\ldots,\mu(T))$: 
\[
	\Prb\left( X(t)\in . \;|\; (\mu(0),\ldots,\mu(T)) \right) = \mu(t)(.) \text{ for every } t\in \{0,\ldots,T\},
\]
which yields the consistency condition.
\end{proof}


\section{Approximate $N$-player correlated equilibria}\label{SectApprox}

The next result shows how to construct a sequence of approximate $N$-player correlated equilibria with approximation error tending to zero as $N\to \infty$ provided we have a correlated solution to the mean field game. The construction can be roughly described as follows: first, the mediator draws some flow of measures from the second marginal of the correlated solution and second, conditioning on such a flow, he draws a sequence of i.i.d.\ recommendations that are privately communicated to the players in the $N$-player games.

In order to rigorously state the result, let $\mathfrak{m}_{0} \in \prbms{\mathcal{X}}$, and let $(\mathfrak{m}^{N})_{N\in \mathbb{N}}$ be such that assumption \hypref{hypInit} holds: $\mathfrak{m}^N = \otimes^{N} \mathfrak{m}_{0,N}$ with $\mathfrak{m}_{0,N} \to \mathfrak{m}_{0}$ as $N\to \infty$.

\begin{thrm} \label{ThApproximation}
Grant \hypref{hypContSystem} and \hypref{hypContCosts}. Suppose that $\rho \in \prbms{\mathcal{R} \times \prbms{\mathcal{X}}^{T+1}}$ is a correlated solution of the mean field game with initial distribution $\mathfrak{m}_{0}$. For $N\in \mathbb{N}$, define $\gamma^{N} \in \prbms{\mathcal{R}^{N}}$ through
\[
	\gamma^{N}(C_{1}\times\ldots\times C_{N}) \doteq \int_{\prbms{\mathcal{X}}^{T+1}} \prod_{i=1}^{N} \rho_{1}(C_{i} \,|\, m)\; \rho_{2}(dm),
\]
where $\rho$ has been factorized according to
\[
	\rho(C\times B) = \int_{B} \rho_{1}(C \,|\, m) \rho_{2}(dm),\quad C\subseteq \mathcal{R},\; B \in \otimes^{T+1}\Borel{\prbms{\mathcal{X}}}.
\]
Then there exists a sequence $(\epsilon_{N})_{N\in \mathbb{N}} \subset [0,\infty)$ such that $\gamma^{N}$ is an $\epsilon_{N}$-correlated equilibrium with initial distribution $\mathfrak{m}^{N}$, for every $N$, and $\epsilon_{N} \to 0$ as $N\to \infty$.
\end{thrm}

\begin{proof}
By symmetry, we may restrict attention to strategy modifications of player~1. For $N\in \mathbb{N}$, set
\[
	\epsilon_N \doteq J^{N}_{1} (\mathfrak{m}^{N}; \gamma^{N}, \Id ) - \inf_{u\in \mathcal{U}} J^{N}_{1} (\mathfrak{m}^{N}; \gamma^{N}, u).
\]
Then $\gamma^{N}$ is an $\epsilon_{N}$-correlated equilibrium with initial distribution $\mathfrak{m}^{N}$. It remains to show that $\epsilon_{N} \to 0$ as $N\to \infty$. To this end, choose a sequence of strategy modifications $u^N$ such that
\[
	J_{1}^{N}(\mathfrak m^N ; \gamma^N , u^N) \leq \inf_{u\in \mathcal{U}} J^{N}_{1}(\mathfrak{m}^N; \gamma^N, u) + \frac{1}{N}.
\]
We have to show that $\lim_{N \to \infty} J^{N}_{1} (\mathfrak{m}^N ; \gamma^N ,\Id) = J(\mathfrak m_0 ; \rho, \Id)$ and that
\[
	\liminf_{N \to \infty} J^{N}_{1} (\mathfrak{m}^N ; \gamma^N , u^N ) \ge J(\mathfrak m_0 ; \rho, \Id),
\]
as this entails that $\epsilon_{N} \to 0$ as $N\to \infty$.

In our setting, the set of strategy modifications, i.e.\ of mappings $u:\mathcal{R} \rightarrow \mathcal{R}$, is finite since $\mathcal{R}$ is finite. Therefore and by the optimality condition, the above limit inferior will be established as soon as $\liminf_{N \to \infty} J^{N}_{1} (\mathfrak m^N ; \gamma^N , u ) \geq J(\mathfrak m_0 ; \rho, u)$ for every strategy modification $u$. It is therefore enough to show that
\begin{equation} \label{ProofApproxCosts}
	\lim_{N \to \infty} J^{N}_{1} (\mathfrak m^N ; \gamma^N , u ) = J(\mathfrak m_0 ; \rho, u)\quad \text{for every } u\in \mathcal{U}.
\end{equation}
We proceed in three steps.

\paragraph{Step One.} For $N\in \mathbb{N}$, set
\[
	\gamma^{N}_{m} \doteq \otimes^{N} \rho_{1}(. \,|\, m),\quad m\in \prbms{\mathcal{X}}^{T+1}.
\]
Then, for every strategy modification $u\in \mathcal{U}$,
\begin{align*}
	J^{N}_{1} (\mathfrak m^N ; \gamma^N , u) &= \int_{\prbms{\mathcal{X}}^{T+1}} J^{N}_{1} (\mathfrak{m}^N ; \gamma^N_{m}, u) \rho_{2}(dm), \\
\intertext{and also}
J(\mathfrak m_0 ; \rho, u) &= \int_{\prbms{\mathcal{X}}^{T+1}} J(\mathfrak m_0 ; \rho_{1}(.\,|\, m) \otimes \delta_{m}, u) \rho_{2}(dm).
\end{align*}
To see this, recall that if we have a realization of $(\mathfrak m^N , \gamma^N , u)$ for player~1 in the $N$-player game or of $(\mathfrak m_0 , \rho, u)$ in the mean field game, then the sequence of noise variables, the initial states (or state), and the random elements realizing the correlated profile $\gamma^N$ (or the correlated flow $\rho$) are independent.

\paragraph{Step Two.} Fix a strategy modification $u$. Let $m\in \prbms{\mathcal{X}}^{T+1}$. As in the proof of Theorem~\ref{ThConvergence}, let $(\Phi_{m},X_{m}(.),\xi(.),\mu_{m}(.))$ and $(\tilde{\Phi}_{m},\tilde{X}_{m}(.),\tilde{\xi}(.),\tilde{\mu}_{m}(.))$ be the distributional limit along a convergent subsequence of realizations of $(\mathfrak{m}^N, \gamma^N_{m}, \Id)$ and $(\mathfrak{m}^N , \gamma^N_{m}, u)$, respectively. Then properties \eqref{ProofConvInitial}--\eqref{ProofConvCosts} there hold, for both $(\Phi_{m},X_{m}(.),\xi(.),\mu_{m}(.))$ and $(\tilde{\Phi}_{m},\tilde{X}_{m}(.),\tilde{\xi}(.),\tilde{\mu}_{m}(.))$. By construction and Step~Three in the proof of Theorem~\ref{ThConvergence}, we also have
\[
	\Prb_{m}\circ \left(\Phi_{m},X_{m}(0),\xi(.),\mu_{m}(.) \right)^{-1} = \tilde{\Prb}_{m}\circ \left(\tilde{\Phi}_{m},\tilde{X}_{m}(0),\tilde{\xi}(.),\tilde{\mu}_{m}(.) \right)^{-1}
\]
and	$\Prb_{m}\circ\, \Phi_{m}^{-1} = \tilde{\Prb}_{m}\circ \tilde{\Phi}_{m}^{-1} = \rho_{1}( . \,|\, m)$. Moreover, by \hypref{hypContSystem}, \hypref{hypInit} and thanks to the fact that $\gamma^{N}_{m}$ is the $N$-fold product of $\rho_{1}( . \,|\, m)$, propagation of chaos holds for the convergent subsequence corresponding to $(\mathfrak{m}^N, \gamma^N_{m}, \Id)$ in the sense that
\[
	\Prb_{m}\circ \left(X_{m}(t),\mu_{m}(t) \right)^{-1} = \hat{m}_{m}(t)\otimes \delta_{\hat{m}_{m}(t)},\quad t\in \{0,\ldots,T\},
\]
for some deterministic flow of measures $\hat{m}_{m} \in \prbms{\mathcal{X}}^{T+1}$ with $\hat{m}_{m}(0) = \mathfrak{m}_{0}$; see, for instance, Theorem~4.2 in \citet{gottlieb98}.
In view of property \eqref{ProofConvLimitDyn}, we therefore have $\Prb_{m}$-almost surely,
\begin{equation} \label{ProofApproxMKV}
\begin{split}
	X_{m}(t+1) &= \Psi\left( t, X_{m}(t), \hat{m}_{m}(t), \Phi_{m}(t,X_{m}(t)), \xi(t+1) \right), \\
	\Prb_{m}\circ X_{m}(t)^{-1} &= \hat{m}_{m}(t),\quad t\in \{0,\ldots,T-1\}.
\end{split}
\end{equation}
Using the independence properties \eqref{ProofConvNoise} and \eqref{ProofConvInd}, we see by induction over the time variable $t$ that Equation~\eqref{ProofApproxMKV}, together with the initial distribution $\Prb_{m}\circ X_{m}(0)^{-1} = \mathfrak{m}_{0}$ and the distribution $\Prb_{m}\circ \Phi_{m}^{-1} = \rho_{1}( . \,|\, m)$, uniquely determines the (deterministic) flow of measures $\hat{m}_{m}$. This can be seen as a uniqueness property for a kind of McKean-Vlasov equation.

\paragraph{Step Three.} We are going to show that $\hat{m}_{m} = m$ for $\rho_{2}$-almost every $m\in \prbms{\mathcal{X}}^{T+1}$, where $\hat{m}_{m}$ is the deterministic flow of measures identified in Step~Two.

Let $((\Omega,\mathcal{F},\Prb^{\ast}),\Phi^{\ast}, X^{\ast}(.), \mu^{\ast}(.), \xi^{\ast}(.))$ be a realization of the triple $(\mathfrak{m}_{0},\rho,\Id)$. The quintuple thus satisfies the dynamics given by Eq.~\eqref{EqLimitDyn}, that is, $\Prb^{\ast}$-almost surely, for every $t\in \{0,\ldots,T-1\}$,
\[
	X^{\ast}(t+1) = \Psi\left( t, X^{\ast}(t), \mu^{\ast}(t), \Phi^{\ast}(t,X^{\ast}(t)), \xi^{\ast}(t+1) \right).
\]
Moreover, $\Prb^{\ast}\circ (\Phi^{\ast},\mu^{\ast}(0),\ldots,\mu^{\ast}(T))^{-1} = \rho$ and $\Prb^{\ast}\circ (X^{\ast}(0))^{-1} = \mathfrak{m}_{0}$. By hypothesis, $\rho$ is a correlated solution with initial distribution $\mathfrak{m}_{0}$. In view of the consistency property, conditioning on $\mu^{\ast}$ therefore yields for $\rho_{2}$-almost every $m\in \prbms{\mathcal{X}}^{T+1}$,
\begin{align*}
	& \Prb^{\ast}\left(X^{\ast}(t)\in . \;|\; \mu^{\ast}=m\right) = m(t)( . ),& &t\in \{0,\ldots,T-1\},& \\
	&\Prb^{\ast}\left(\Phi^{\ast}\in . \;|\; \mu^{\ast}=m \right) = \rho_{1}( . \,|\,m).& &&
\end{align*}
Since the noise variables $\xi^{\ast}(t)$, $t\in \{1,\ldots,T\}$, are i.i.d.\ with common distribution $\nu$ and $\xi^{\ast}(.)$, $X^{\ast}(0)$, and $(\Phi^{\ast}, \mu^{\ast}(.))$ are independent, it follows that for $\rho_{2}$-almost every flow $m\in \prbms{\mathcal{X}}^{T+1}$, the triple $(X^{\ast}(.), \Phi^{\ast}, \xi^{\ast}(.))$ solves Equation~\eqref{ProofApproxMKV} $\Prb^{\ast}\left( . \;|\; \mu^{\ast}=m\right)$-almost surely with deterministic flow of measures $\hat{m}_{m} = m$. Uniqueness of solutions for Equation~\eqref{ProofApproxMKV} now entails that
\begin{align*}
	& \hat{m}_{m} = m & &\text{and}& &\Prb_{m}\circ \left(\Phi_{m},\mu_{m} \right)^{-1} = \tilde{\Prb}_{m}\circ \left(\tilde{\Phi}_{m},\tilde{\mu}_{m} \right)^{-1} = \rho_{1}(. |m)\otimes \delta_{m}&
\end{align*}
for $\rho_{2}$-almost every $m\in \prbms{\mathcal{X}}^{T+1}$. This also shows that, given a ($\rho_{2}$-typical) flow of measures $m\in \prbms{\mathcal{X}}^{T+1}$, any convergent subsequence of realizations of $(\mathfrak{m}^N, \gamma^N_{m}, \Id)$ has the same limit in distribution, and analogously for realizations of $(\mathfrak{m}^N ; \gamma^N_{m}, u)$.

Convergence of costs according to property \eqref{ProofConvCosts} and integration against $\rho_{2}$ according to Step~One, in conjunction with dominated convergence, finish the proof. 
\end{proof}

\begin{appendix}

\section{Auxiliary results}\label{AppAux}

Here, we collect some auxiliary results, mostly elementary, regarding weak convergence and exchangeable triangular arrays. We refer to \citet{billingsley68} for the theory of weak convergence of probability measures.

Let $\mathcal{Y}$, $\mathcal{Z}$ be Polish spaces. For $n\in \mathbb{N}$, let $Y_{n}$, $Z_{n}$ be random variables on $(\Omega_{n},\mathcal{F}_{n},\Prb_{n})$ with values in $\mathcal{Y}$ and $\mathcal{Z}$, respectively.

\begin{lemma} \label{LemmaEquationConv}
	Let $\Psi\!: \mathcal{Z} \rightarrow \mathcal{Y}$ be measurable.  Suppose that $(Y_{n},Z_{n})$ converges in distribution to $(Y,Z)$ as $n\to \infty$ for some $\mathcal{Y}\times \mathcal{Z}$-valued random variable $(Y,Z)$ defined on $(\Omega,\mathcal{F},\Prb)$.
	
	If $Y_{n} = \Psi(Z_{n})$ $\Prb_{n}$-almost surely for every $n\in \mathbb{N}$ and if $\Psi$ is continuous $\Prb\circ Z^{-1}$-almost everywhere, then $Y = \Psi(Z)$ $\Prb$-almost surely.
\end{lemma}

\begin{proof}
The hypothesis that $\Psi$ is continuous $\Prb\circ Z^{-1}$-almost everywhere implies that the mapping $\mathcal{Y}\times \mathcal{Z} \ni (y,z) \mapsto (y,\Psi(z))$ is continuous $\Prb\circ (Y,Z)^{-1}$-almost everywhere. By the convergence assumption and the mapping theorem \citet[Theorem~I.5.1]{billingsley68}, it follows that
\[
	(Y_{n},\Psi(Z_{n})) \stackrel{n\to\infty}{\longrightarrow} (Y,\Psi(Z)) \text{ in distribution}.
\]
Let $D\doteq \{ (y,\tilde{y})\in \mathcal{Y}\times \mathcal{Y} : y = \tilde{y} \}$ be the diagonal in $\mathcal{Y}\times \mathcal{Y}$. Then $D$ is closed in $\mathcal{Y}\times \mathcal{Y}$, hence
\[
	\limsup_{n\to\infty} \Prb_{n} \left( (Y_{n},\Psi(Z_{n}))\in D \right) \leq \Prb \left( (Y,\Psi(Z))\in D \right) 
\]
by the Portmanteau theorem \citep[Theorem~I.2.1]{billingsley68}. On the other hand, we have $\Prb_{n} \left( (Y_{n},\Psi(Z_{n}))\in D \right) = 1$ for every $n\in \mathbb{N}$ since $Y_{n} = \Psi(Z_{n})$ $\Prb_{n}$-almost surely by hypothesis. It follows that $\Prb \left( (Y,\Psi(Z))\in D \right) = 1$, that is, $Y = \Psi(Z)$ $\Prb$-almost surely.
\end{proof}

\begin{lemma} \label{LemmaPerturbConv}
	Let $\mathrm{d}_{\mathcal{Z}}$ be a metric compatible with the topology of $\mathcal{Z}$. Let $(Y_n,Z_n, \tilde Z_n)_{n \in \mathbb N}$ be a sequence of $\mathcal Y \times \mathcal Z \times \mathcal Z$-valued random variables, where each $(Y_n,Z_n,\tilde Z_n)$ is defined on some probability space $(\Omega_{n},\mathcal{F}_{n},\Prb_{n})$, $n \in \mathbb N$.   
	
Suppose that $(Y_{n},Z_{n})$ converges in distribution to $(Y,Z)$ as $n\to \infty$ for some $\mathcal{Y}\times \mathcal{Z}$-valued random variable $(Y,Z)$ defined on $(\Omega,\mathcal{F},\Prb)$, and that
	\[
		\Mean_{n}\left[ \mathrm{d}_{\mathcal{Z}}(Z_{n},\tilde{Z}_{n}) \right] \stackrel{n\to\infty}{\longrightarrow} 0.
	\]
	Then $(Y_{n},\tilde{Z}_{n})$ converges in distribution to $(Y,Z)$ as $n\to \infty$.
\end{lemma}

\begin{proof}
Let $\mathrm{d}_{\mathcal{Y}}$ be any metric compatible with the topology of $\mathcal{Y}$. Set
\[
	\mathrm{d}\left((y,z),(\tilde{y},\tilde{z})\right) \doteq \mathrm{d}_{\mathcal{Y}}(y,\tilde{y}) + \mathrm{d}_{\mathcal{Z}}(z,\tilde{z}),\quad (y,z),\; (\tilde{y},\tilde{z}) \in \mathcal{Y}\times \mathcal{Z}.
\]
Then $\mathrm{d}$ is a metric on $\mathcal{Y}\times \mathcal{Z}$ compatible with the product topology. By hypothesis and Markov's inequality, we have that $(\mathrm{d}((Y_{n},Z_{n}),(Y_{n},\tilde{Z}_{n})))_{n\in \mathbb{N}}$ converges to zero in probability. As the limit is a constant, this is equivalent to convergence in distribution, and the underlying probability spaces may depend on $n\in \mathbb{N}$. The assertion now follows from Theorem~I.4.1 in \citet[p.\,25]{billingsley68}.
\end{proof}

For the next result, let $\kappa_{n}$ be a regular conditional distribution of $Y_{n}$ given the $\sigma$-algebra generated by $Z_{n}$, each $n\in \mathbb{N}$. Thus, $\kappa_{n}$ is a mapping $\Omega_{n}\times \Borel{\mathcal{Y}} \rightarrow [0,1]$ that induces a $\prbms{\mathcal{Y}}$-valued random variable and is such that, for every $A\in \Borel{\mathcal{Y}}$,
\[
	\Prb_{n}\left( Y_{n} \in A \;|\; Z_{n} \right) = \kappa_{n}(A) \quad \Prb_{n}\text{-almost surely.}
\]
A regular conditional distribution of $Y_{n}$ given $Z_{n}$ exists since $\mathcal{Y}$ is a Polish space, and it is uniquely determined with probability one when seen as a $\prbms{\mathcal{Y}}$-valued random variable.

\begin{lemma} \label{LemmaCondDistConv}
	Let $\kappa_{n}$ be a regular conditional distribution of $Y_{n}$ given $\sigma(Z_{n})$ as above. Suppose that $(Y_{n},Z_{n},\kappa_{n})$ converges in distribution to $(Y,Z,\kappa)$ as $n\to \infty$ for some $\mathcal{Y}\times \mathcal{Z}\times \prbms{\mathcal{Y}}$-valued random variable $(Y,Z,\kappa)$ defined on $(\Omega,\mathcal{F},\Prb)$. Then, for every $A\in \Borel{\mathcal{Y}}$, every $B\in \Borel{\mathcal{Z}}$,
	\[
		\Prb\left( Y \in A, Z\in B \right) = \Mean\left[ \mathbf{1}_{B}(Z)\cdot \kappa(A)\right]. 
	\]
	If, in addition, $\kappa(A)$ is $\sigma(Z)$-measurable for every $A\in \Borel{\mathcal{Y}}$, then $\kappa$ is a regular conditional distribution of $Y$ given $\sigma(Z)$.
\end{lemma}

\begin{proof}
Let $Q\in \prbms{\mathcal{Y}\times \mathcal{Z}}$ be the joint law of $Y$ and $Z$: $Q\doteq \Prb\circ (Y,Z)^{-1}$. Define another measure $\tilde{Q}\in \prbms{\mathcal{Y}\times \mathcal{Z}}$ by setting, for $A\in \Borel{\mathcal{Y}}$, $B\in \Borel{\mathcal{Z}}$,
\[
	\tilde{Q}(A\times B) \doteq \Mean\left[ \mathbf{1}_{B}(Z)\cdot \kappa(A) \right].
\]

Let $g\!: \mathcal{Y}\times \mathcal{Z} \rightarrow \mathbb{R}$ be bounded and measurable. Then, for every $n\in \mathbb{N}$,
\[
	\Mean_{n}\left[ g(Y_{n},Z_{n}) \right] = \Mean_{n}\left[ \int_{\mathcal{Y}} g(y,Z_{n}) \kappa_{n}(dy) \right],
\]
since $\kappa_{n}$ is a version of the regular conditional distribution of $Y_{n}$ given $\sigma(Z_{n})$ by hypothesis. If $g$ is bounded and continuous, then, by convergence in distribution of $(Y_{n},Z_{n})$ to $(Y,Z)$,
\[
	\lim_{n\to\infty} \Mean_{n}\left[ g(Y_{n},Z_{n}) \right] = \Mean\left[ g(Y,Z) \right],
\]
but also, by convergence in distribution of $(Z_{n},\kappa_{n})$ to $(Z,\kappa)$,
\[
	\lim_{n\to\infty} \Mean_{n}\left[ \int_{\mathcal{Y}} g(y,Z_{n}) \kappa_{n}(dy) \right] = \Mean\left[ \int_{\mathcal{Y}} g(y,Z) \kappa(dy) \right]
\]
since the mapping $(z,m) \mapsto \int_{\mathcal{Y}} g(y,z) m(dy)$ is bounded and continuous on $\mathcal{Z}\times \prbms{\mathcal{Y}}$ if $g$ is bounded and continuous on $\mathcal{Y}\times \mathcal{Z}$; cf.\ Theorem~I.5.5 in \citet[p.\,34]{billingsley68}.

Therefore, for every $g\!: \mathcal{Y}\times \mathcal{Z} \rightarrow \mathbb{R}$ bounded and continuous,
\[
	\int_{\mathcal{Y}\times \mathcal{Z}} g\, dQ = \Mean\left[ g(Y,Z) \right] = \Mean\left[ \int_{\mathcal{Y}} g(y,Z) \kappa(dy) \right] = \int_{\mathcal{Y}\times \mathcal{Z}} g\, d\tilde{Q}.
\]
A measure on the Borel sets of a Polish space is uniquely determined by its integrals over all bounded continuous functions. It follows that $Q = \tilde{Q}$. This in turn implies that for all $A\in \Borel{\mathcal{Y}}$, all $B\in \Borel{\mathcal{Z}}$,
\[
	\Mean\left[ \mathbf{1}_{A}(Y) \cdot \mathbf{1}_{B}(Z) \right] = \Mean\left[ \mathbf{1}_{B}(Z)\cdot \kappa(A)\right],
\]
which yields the first part of the assertion. If, in addition, $\kappa(A)$ is $\sigma(Z)$-measurable for every $A\in \Borel{\mathcal{Y}}$, then
\[
	\Prb\left( Y \in A \;|\; Z \right) = \kappa(A) \quad \Prb\text{-almost surely}
\]
by the above property and the definition of conditional expectation.
\end{proof}

\begin{lemma} \label{LemmaCondDistTower}
	Let $Y$ and $\kappa$ be random variables on some probability space $(\Omega,\mathcal{F},\Prb)$ with values in $\mathcal{Y}$ and $\prbms{\mathcal{Y}}$, respectively. Let $\mathcal{C}$, $\mathcal{G}$, $\tilde{\mathcal{C}}$ be sub-$\sigma$-algebras of $\mathcal{F}$ such that $\mathcal{C} \subseteq \mathcal{G} \subseteq \tilde{\mathcal{C}}$.

	If $\kappa$ is a regular conditional distribution of $Y$ given $\mathcal{C}$ as well as given $\tilde{\mathcal{C}}$, then $\kappa$ is also a regular conditional distribution of $Y$ given $\mathcal{G}$.
\end{lemma}

\begin{proof}
Suppose that $\kappa$ is a regular conditional distribution of $Y$ given $\mathcal{C}$ as well as given $\tilde{\mathcal{C}}$. Let $A\in \Borel{\mathcal{Y}}$. The first part of the assumption implies that $\kappa(A)$ is $\mathcal{C}$-measurable, hence also $\mathcal{G}$-measurable (since $\mathcal{C} \subseteq \mathcal{G}$). The second part of the assumption entails that
\[
	\Prb\left( \{Y \in A \} \cap C \right) = \Mean\left[ \mathbf{1}_{C} \cdot \kappa(A) \right]
\]
for all $C\in \tilde{\mathcal{C}}$, hence also for all $C\in \mathcal{G}$ (since $\mathcal{G} \subseteq \tilde{\mathcal{C}}$). This shows that $\kappa$ is a regular conditional distribution of $Y$ given $\mathcal{G}$.
\end{proof}

The next result recalls the conditional distribution of an element of a finite exchangeable sequence given the associated empirical measure.

\begin{lemma} \label{LemmaCondDistExchange}
	Let $Y_{1},\ldots,Y_{N}$ be a finite exchangeable sequence of $\mathcal{Y}$-valued random variables on some probability space $(\Omega,\mathcal{F},\Prb)$, and let
	\[
		\mu^{N}_{\omega} \doteq \frac{1}{N} \sum_{i=1}^{N} \delta_{Y_{i}(\omega)},\quad \omega \in \Omega,
	\]
	be the associated empirical measure. Then, for every $i\in \{1,\ldots,N\}$,
	\[
		\Prb\left( Y_{i}\in . \;|\; \mu^{N} \right) = \mu^{N}(.)
	\]
	in the sense that $\mu^{N}$ is a regular conditional distribution of $Y_{i}$ given (the $\sigma$-algebra generated by) $\mu^{N}$.
\end{lemma}

\begin{proof}
The assertion follows from Lemma~11.11 in \citet[p.\,213]{kallenberg01}.
\end{proof}

\end{appendix}

\bibliographystyle{abbrvnat}

\end{document}